\documentclass[a4paper,UKenglish,cleveref, autoref, thm-restate]{lipics-v2021}

\hideLIPIcs  

\usepackage{todonotes}

\newcommand{\N}{\mathbb N}
\newcommand{\Q}{\mathbb Q}
\newcommand{\bA}{{\mathfrak A}}
\newcommand{\bB}{{\mathfrak B}}
\newcommand{\bC}{{\mathfrak C}}
\newcommand{\bD}{{\mathfrak D}}

\newcommand{\bG}{{\mathfrak G}}
\newcommand{\bH}{{\mathfrak H}}
\newcommand{\bT}{{\mathfrak T}}

\DeclareMathOperator{\Aut}{Aut}
\DeclareMathOperator{\End}{End}

\DeclareMathOperator{\fPol}{fPol}

\DeclareMathOperator{\VCSP}{VCSP}
\DeclareMathOperator{\CSP}{CSP}

\DeclareMathOperator{\Pol}{Pol}

\DeclareMathOperator{\Opt}{Opt}
\DeclareMathOperator{\Id}{Id}
\DeclareMathOperator{\Feas}{Feas}

\DeclareMathOperator{\OIT}{OIT}

\DeclareMathOperator{\multigr}{Multigraph}

\newcommand{\ignore}[1]{}

\title{The Complexity of Resilience for Digraph Queries} 

\author{Manuel Bodirsky}{Institut f\"ur Algebra, TU Dresden, Germany}{manuel.bodirsky@tu-dresden.de}{https://orcid.org/0000-0001-8228-3611}{The author has been funded by the European Research Council (Project POCOCOP, ERC Synergy Grant 101071674) and by the DFG (Project FinHom, Grant 467967530). Views and opinions expressed are however those of the authors only and do not necessarily reflect those of the European Union or the European Research Council Executive Agency. Neither the European Union nor the granting authority can be held responsible for them.}

\author{\v{Z}aneta Semani\v{s}inov\'{a}}{Institute of Discrete Mathematics and Geometry, TU Wien, Austria}{zaneta.semanisinova@tu-dresden.de}{https://orcid.org/0000-0001-8111-0671}{The author has been funded by the European Research Council (Project POCOCOP, ERC Synergy Grant 101071674). Views and opinions expressed are however those of the authors only and do not necessarily reflect those of the European Union or the European Research Council Executive Agency. Neither the European Union nor the granting authority can be held responsible for them. This research was funded in whole or in part by the Austrian Science Fund (FWF) 10.55776/ESP6949724.}

\authorrunning{M. Bodirsky and \v{Z}. Semani\v{s}inov\'{a}} 

\Copyright{M. Bodirsky and \v{Z}. Semani\v{s}inov\'{a}} 



\category{} 

\relatedversion{} 

\nolinenumbers 

\EventEditors{John Q. Open and Joan R. Access}
\EventNoEds{2}
\EventLongTitle{42nd Conference on Very Important Topics (CVIT 2016)}
\EventShortTitle{CVIT 2016}
\EventAcronym{CVIT}
\EventYear{2016}
\EventDate{December 24--27, 2016}
\EventLocation{Little Whinging, United Kingdom}
\EventLogo{}
\SeriesVolume{42}
\ArticleNo{23}

\begin{document}

\maketitle

\begin{abstract}
We prove a complexity dichotomy for the resilience problem for unions of conjunctive digraph queries (i.e., for existential positive sentences over the signature $\{R\}$ of directed graphs). Specifically, for every union $\mu$ of conjunctive digraph queries, the following problem is in P or NP-complete: given a directed multigraph $G$ and a natural number $u$, can we remove $u$ edges from $G$ so that $G \models \neg \mu$? In fact, we verify a more general dichotomy conjecture from~\cite{Resilience-VCSPs} for all resilience problems in the special case of directed graphs, and show that for such unions of queries $\mu$ there exists a countably infinite (`dual') valued structure $\Delta_\mu$ which either primitively positively constructs 1-in-3-3-SAT, and hence the resilience problem for $\mu$ is NP-complete by general principles, or has a pseudo cyclic canonical fractional polymorphism, and the resilience problem for $\mu$ is in P.
\end{abstract}

\newpage

\section{Introduction}\label{sec:intro}
The \emph{resilience problem} for a fixed conjunctive query,
or more generally for a union of conjunctive queries $\mu$, is the problem of deciding for a given database $\bA$ and $u \in {\mathbb N}$ whether it is possible to remove at most $u$ tuples from $\bA$  so that $\bA$ does not satisfy $\mu$. 
The resilience problem lies at the core of algorithmic challenges in various forms of reverse data management, where an action is required on the input data to achieve a desired outcome in the output data~\cite{RDM}. 
The computational complexity of this problem depends on the query $\mu$. The resilience problem is always in NP, and often NP-complete, but for some queries $\mu$ the problem can be solved in polynomial time; see, e.g.,~\cite{Resilience,NewResilience,LatestResilience, Resilience-VCSPs} for some partial classification results.

The computational complexity of the problem also depends on
whether we view the database $\bA$ under set semantics (i.e., $\bA$ is treated as a relational structure) or under bag semantics (i.e., $\bA$ is a structure where each tuple appears with some multiplicity), and both settings have been studied in the literature (in particular, see~\cite{Resilience-VCSPs,LatestResilience} for results in bag semantics). The importance of bag semantics stems from applications: bag databases represent SQL databases more faithfully. There are examples of conjunctive queries $\mu$ for which the resilience problem in bag and set semantics have different complexities~\cite{LatestResilience}.

Recently, a connection between the resilience problem under bag semantics and valued constraint satisfaction has been established~\cite{Resilience-VCSPs}. 
The
connection is 
based on the fact that for every union of connected conjunctive queries $\mu$, the resilience problem in bag semantics is equal to a valued constraint satisfaction problem (VCSP) for some template dependent on $\mu$, and therefore the algebraic tools developed for describing the complexity of VCSPs can be utilized. To do so, we focus in this paper on the resilience problem exclusively under bag semantics and from now on, always implicitly assume this semantics for resilience problems.

It has been conjectured that resilience problems exhibit a complexity dichotomy in the sense that all problems are NP-complete or in P~\cite{LatestResilience}.
This conjecture has been verified in some 
special cases, 
for instance if $\mu$ is a conjunctive query which is self-join-free~\cite{LatestResilience}, 
or if $\mu$ is a union of conjunctive queries that are Berge-acyclic~\cite{Resilience-VCSPs}. 
The proof of the latter is based on a connection to finite-domain VCSPs, which also covers resilience problems for regular path queries (RPQs), and even two-way RPQs.
Resilience of (one-way) RPQs has also been studied recently in~\cite{RPQ} where the authors present language-theoretic conditions for computational hardness.
However, the conjecture in full generality remains open.

In this article, we confirm the complexity dichotomy conjecture in the special case where the signature of the database consists of a single binary relation symbol $R$, 
that is, we prove the following:
\begin{theorem}\label{thm:main-informal}
If $\mu$ is a union of conjunctive queries over a binary signature $\{R\}$, then the resilience problem for $\mu$ is in P or NP-complete.
\end{theorem}

The class of unions of conjunctive queries over $\{R\}$
is incomparable with the class of queries studied in~\cite{NewResilience} (in set semantics), since they study queries with arbitrary signatures, but with a single repetition of a single binary relation symbol.
In the case that the signature is equal to $\{R\}$, the query expresses a directed graph property and the resilience problem can be phrased as follows: given a directed multigraph $G$ and a natural number $u$, can we remove $u$ edges from $G$ so that $G \models \neg \mu$?
Edge-removal problems have been studied from a computational complexity perspective in the graph theory community as well, especially for concrete properties~\cite[Section A1.2]{GareyJohnson}.
In~\cite{FominEtAl20} the authors study edge-removal problems for first-order logic properties in general; however, they only consider simple undirected graphs and study the problem from the perspective of \emph{parametrized complexity}, where the number of edges that is removed is the parameter.

The scope of our contribution extends beyond verifying the complexity dichotomy conjecture for digraph resilience problems: we also verify a variant of a stronger conjecture (from~\cite{Resilience-VCSPs}) which provides a precise mathematical condition aiming at predicting the border between NP-hardness and polynomial-time tractability,  
based on simulations of a hard Boolean constraint satisfaction problem (CSP) using so-called \emph{pp-constructions}. This condition is one-sided correct in the sense that if it applies, the corresponding resilience problem is NP-hard. The authors of~\cite{Resilience-VCSPs} conjectured that if the condition does not apply, the resilience problem is in P.

Several results in the present paper are relevant for the larger research goal of classifying the complexity of all resilience problems in bag semantics by modeling them as VCSPs and applying methods and results from the VCSP literature. For instance, our result that the two conditions of the dichotomy statement are disjoint (Corollary~\ref{cor:disjoint}) holds for resilience problems in general (without the assumption that P $\neq$ NP). Another result that holds for resilience in general is Theorem~\ref{thm:self-join-var}, which provides pp-constructions (and, therefore, polynomial-time reductions) based on the idea of \emph{self-join variations} from~\cite{NewResilience}.
We believe that this paper is an important step towards classifying the complexity for resilience problems of queries with self joins and understanding reductions between resilience problems by algebraic and logic tools. 

\section{Preliminaries}

In this section, we provide preliminaries that cover the notions appearing in Section~\ref{sect:main}, where the main theorem of the article (Theorem~\ref{thm:main-ucq}) is stated. Since the theorem provides not only a complexity dichotomy, but also an algebraic one, this requires several notions from the theory of VCSPs. For readers mostly interested in the complexity of resilience problems on its own, we recommend reading only Sections~\ref{sect:vs}--\ref{sect:connection} and skipping Sections~\ref{sect:expr}--\ref{sect:fpol}, and coming back to them 
when they are needed in the proofs in the article.

The set $\{0,1,2,\dots\}$ of natural numbers is denoted by ${\mathbb N}$. For $k \in \N$, the set $\{1,\dots, k\}$ will be denoted by $[k]$. The set of rational numbers is denoted by $\mathbb Q$ and the standard strict linear order on $\Q$ by $<$. The set of real numbers is denoted by $\mathbb{R}$. 
We also need an additional value $\infty$; all
we need to know about $\infty$ is that
\begin{itemize}
\item $a < \infty$ for every $a \in {\mathbb R}$,
\item $a + \infty = \infty + a = \infty$ for all $a \in {\mathbb R} \cup \{\infty\}$, and 
\item $0 \cdot \infty = 
\infty \cdot 0 = 0$
and $a \cdot \infty =
\infty \cdot a = \infty$ for $a > 0$. 
 \end{itemize}

Let $A$ be a set and $k \in \N$.
If $t \in A^k$, then we implicitly assume that $t=(t_1, \dots, t_k)$, where $t_1,\dots,t_k \in A$. 
If $\ell \in \N$ and $f \colon A^{\ell} \to A$ is an operation on $A$ and $t^1, \dots,t^\ell \in A^k$, then we denote 
$(f(t^1_1, t^2_1, \dots, t^\ell_1), \ldots, f(t^1_k, t^2_k, \dots, t^\ell_k))$
by $f(t^1, \dots, t^\ell)$ and say that 
\emph{$f$ is applied componentwise}.

\subsection{Valued structures} \label{sect:vs}
Let $C$ be a set and let $k \in {\mathbb N}$. 
A \emph{valued relation of arity $k$ over $C$}
is 
a function $R \colon C^k \to {\mathbb Q} \cup \{\infty\}$.  
We write ${\mathscr R}_C^{(k)}$ for the set of all valued relations over $C$ of arity $k$, and define 
$${\mathscr R}_C := \bigcup_{k \in {\mathbb N}} {\mathscr R}_C^{(k)}.$$
A valued relation is called \emph{finite-valued} if it takes values only in $\Q$.
Usual relations will also be called \emph{crisp} relations. A valued relation $R \in {\mathscr R}_C^{(k)}$ that only takes values from $\{0,\infty\}$ will be identified with the crisp relation 
$\{t \in C^k \mid R(t) = 0\}.$
The unary empty relation, where every element evaluates to $\infty$, is denoted by $\bot$. The crisp equality relation, where a pair of elements evaluates to $0$ if they are equal and evaluates to $\infty$ otherwise, is denoted by $(=)_0^\infty$. 
For $R \in {\mathscr R}_C^{(k)}$ the \emph{feasibility relation of $R$} is defined as
$\Feas(R) := \{t \in C^k \mid R(t) < \infty\}.$
%

A \emph{(relational) signature} $\tau$ is a~set of \emph{relation symbols}, each of them equipped with an arity from~${\mathbb N}$. A~\emph{valued $\tau$-structure} $\Gamma$ consists of a~set $C$, which is also called the \emph{domain} of $\Gamma$, and a~valued relation $R^{\Gamma} \in {\mathscr R}_C^{(k)}$ for each relation symbol $R \in \tau$ of arity $k$. All valued structures in this article have countable domains. We often write $R$ instead of $R^{\Gamma}$ if the valued structure is clear from the context.
A valued $\tau$-structure where all valued relations only take values from $\{0,\infty\}$
may be viewed as a~\emph{relational} or \emph{crisp} $\tau$-structure in the classical sense. 
When not specified, we assume that the domains of relational structures $\bA, \bB, \dots$ are denoted $A, B, \dots$, respectively, and the domains of valued structures $\Gamma, \Delta, \dots$ are denoted $C, D, \dots$, respectively.

\begin{example}\label{expl:max-cut-vs}
Let $R$ be a binary relation symbol. Then $\Gamma_{\textup{MC}}$ with the domain $\{0,1\}$ and the signature $\{R\}$ where $R^{\Gamma_{\textup{MC}}}(x,y) = 0$ if $x = 0$ and $y =1$, and $R^{\Gamma_{\textup{MC}}}(x,y) = 1$ otherwise, is a valued structure.
\end{example}
If $\sigma \subseteq \tau$ and $\Gamma'$ is a valued $\sigma$-structure such that $R^{\Gamma'} = R^{\Gamma}$ for every $R \in \sigma$, then we call $\Gamma'$ a \emph{reduct} of $\Gamma$ and $\Gamma$ an \emph{expansion} of $\Gamma'$.

Let $\tau$ be a relational signature.
A first-order formula is called atomic if it is of the form $R(x_1, \dots, x_k)$ for some $R \in \tau$ of arity $k$, $x=y$, or $\bot$.
We introduce a generalization of conjunctions of atomic formulas to the valued setting. 
An \emph{atomic $\tau$-expression} is an expression
of the form $R(x_1,\dots,x_k)$ for $R \in \tau \cup \{(=)_0^\infty, \bot\}$ and (not necessarily distinct) variable symbols $x_1,\dots,x_k$. 
A \emph{$\tau$-expression} is an expression $\phi$ 
of the form 
$\sum_{i=1}^m \phi_i$
where $m \in {\mathbb N}$ 
and $\phi_i$ for $i \in \{1,\dots,m\}$ is 
an atomic $\tau$-expression.
Note that the same atomic $\tau$-expression might appear several times in the sum. 
We write $\phi(x_1,\dots,x_n)$ for a $\tau$-expression where all the variables  
are from the set $\{x_1,\dots,x_n\}$. 
If $\Gamma$ is a valued $\tau$-structure, then a $\tau$-expression $\phi(x_1,\dots,x_n)$ defines over $\Gamma$ a member of ${\mathscr R}_C^{(n)}$ in a natural way, which we denote by $\phi^{\Gamma}$. If $\phi$ is the empty sum then $\phi^{\Gamma}$ is constant~$0$. 

\begin{definition} 
Let $k \in {\mathbb N}$, 
let $R \in {\mathscr R}^{(k)}_C$, and let $\alpha$ be a permutation of $C$. Then $\alpha$ \emph{preserves} $R$ if for all $t \in C^k$
we have $R(\alpha(t)) = R(t)$. 
If $\Gamma$ is a valued structure with domain $C$, then 
an \emph{automorphism} of $\Gamma$ is a permutation of $C$ that preserves all valued relations of $\Gamma$. 
\end{definition}
The set of all automorphisms of $\Gamma$ is denoted by $\Aut(\Gamma)$, and forms a group with respect to composition. 

Let $A$ be a~set and $R \subseteq A^k$. An operation $f \colon A^\ell \to A$ on the set $A$ \emph{preserves} $R$ if $f(t^1, \dots, t^\ell) \in R$ for every
$t^1,\dots, t^\ell \in R$.
If $\bA$ is a~relational structure and $f$ preserves all relations of $\bA$, then $f$ is called a~\emph{polymorphism} of $\bA$. The set of all polymorphisms of $\bA$ is denoted by $\Pol(\bA)$ and is closed under composition. We write $\Pol^{(\ell)}(\bA)$ for the set of $\ell$-ary operations in $\Pol(\bA)$, $\ell \in \N$. Unary polymorphisms are called \emph{endomorphisms} and $\Pol^{(1)}(\bA)$ is also denoted by $\End(\bA)$. 

Let $\tau$ be a relational signature and let $\bA$ and $\bB$ be relational $\tau$-structures. A map $h \colon A \to B$ is called a \emph{homomorphism} from $\bA$ to $\bB$ if for every $R \in \tau$ of arity $k$ and every $t \in R^\bA$, $h(t)\in R^\bB$. $\bA$ and $\bB$ are called \emph{homomorphically equivalent} if there is a homomorphism from $\bA$ to $\bB$ and from $\bB$ to $\bA$, and they are called \emph{homomorphically incomparable} if there is no homomorphism from $\bA$ to $\bB$ or from $\bB$ to $\bA$.
The generalizations of the notions of polymorphisms and homomorphisms to valued structures will be defined in Sections~\ref{sect:pp-const} and~\ref{sect:fpol}.

\subsection{Valued constraint satisfaction problems}

In this section we assume that
 $\Gamma$ is a fixed valued $\tau$-structure for a 
 \emph{finite} signature $\tau$. We first define the valued constraint satisfaction problem of a relational structure and then explain the connection to the less general constraint satisfaction problem 
in Remark~\ref{rem:crisp-expr}.


\begin{definition} \label{def:vcsp}
The \emph{valued constraint satisfaction problem for $\Gamma$}, denoted by \emph{$\VCSP(\Gamma)$}, is the computational
problem to decide for a given $\tau$-expression $\phi(x_1,\dots,x_n)$ 
and a given $u \in {\mathbb Q}$
whether there exists
$t \in C^n$ such that $\phi^{\Gamma}(t) \leq u$. 
We refer to $\phi(x_1,\dots,x_n)$ as an \emph{instance} of $\VCSP(\Gamma)$,
and to $u$ as the \emph{threshold}. 
%
Tuples $t \in C^n$ such that $\phi^{\Gamma}(t) \leq u$ are called a \emph{solution for $(\phi,u)$}.
The \emph{cost}  of $\phi$ (with respect to $\Gamma$) is defined to be
$$\inf_{t \in C^n} \phi^{\Gamma}(t).$$
\end{definition}

In some contexts, it is beneficial to consider only a given $\tau$-expression $\phi$ to be the input of $\VCSP(\Gamma)$ (rather than $\phi$ and the threshold $u$) and a tuple $t \in C^n$ is then called a \emph{solution for $\phi$} if the cost of $\phi$ equals $\phi^{\Gamma}(t)$.
Note that in general there might not be any solution; however, this is never the case for VCSPs considered in this paper as they stem from resilience problems.
If there exists a tuple $t \in C^n$ such that $\phi^{\Gamma}(t) < \infty$ then $\phi$ is called \emph{satisfiable}.


For relational structures, VCSPs specialize to CSPs, as explained below.

\begin{remark}\label{rem:crisp-expr}
If $\bA$ is a relational $\tau$-structure, then $\CSP(\bA)$ is the problem of deciding satisfiability of conjunctions of atomic formulas over $\tau$ in $\bA$.
Note that for every $\tau$-expression $\phi(x_1, \dots, x_n)$, $\phi^{\bA}$ defines a crisp relation and can be viewed as a conjunction of atomic formulas, which defines the same relation. Minimizing $\phi^{\bA}$ then corresponds to finding $t \in A^n$ such that $\phi^{\bA}(t)=0$, i.e. $t$ that satisfies all atomic formulas in the conjunction. Therefore, $\VCSP(\bA)$ and $\CSP(\bA)$ are essentially the same problem.
\end{remark}

\begin{example} \label{expl:max-cut-vcsp}
The problem $\VCSP(\Gamma_{\textup{MC}})$ for the valued structure $\Gamma_{\textup{MC}}$ from Example~\ref{expl:max-cut-vs} models the \emph{directed max-cut} problem: given a finite directed multigraph $(V,E)$,
find a partition of the vertices $V$ into two classes $A$ and $B$ such that the number of edges from $A$ to $B$ is maximal. Maximising the number of edges from $A$ to $B$ amounts to minimising the number $e$ of edges within $A$, within $B$, and from $B$ to $A$. So when we associate $A$ to the preimage of $0$
and $B$ to the preimage of $1$, computing the answer corresponds to finding the evaluation map $f \colon V \rightarrow \{0,1\}$  that minimises the value 
\[\sum_{(x,y) \in E} {R^{\Gamma_{\textup{MC}}}(f(x), f(y))},\]
which can be formulated as an instance of 
$\VCSP(\Gamma_{\textup{MC}})$. Conversely, every instance of $\VCSP(\Gamma_{\textup{MC}})$ corresponds to a directed max-cut instance.   
\end{example}

\begin{example} \label{expl:oit}
Consider the relation $\OIT = \{(0,0,1),(0,1,0),(1,0,0) \}$. 
$\CSP(\{0,1\};\OIT)$ is the so called 1-in-3-3-SAT problem, which is known to be NP-complete (see, e.g., \cite[Example 1.2.2]{Book}).
\end{example}

\subsection{Conjunctive queries and resilience} \label{sect:cq}
A first-order formula is called \emph{primitive positive} if it is an existentially quantified conjunction of atomic formulas.
A \emph{conjunctive query} over a (relational) signature $\tau$ is a primitive positive $\tau$-sentence and 
a \emph{union of conjunctive queries} is a (finite) disjunction of conjunctive
queries. 
Note that every existential positive sentence can be written as a union of conjunctive queries.

If $\bA$ is a relational $\tau$-structure and $\mu$ is a union of conjunctive queries over $\tau$ with a quantifier-free part $\mu'(v_1, \dots, v_n)$, we say that $\alpha \colon \{v_1, \dots, v_n\} \to A$ \emph{witnesses that $\bA \models \mu$} if $\bA \models \mu'(\alpha(v_1), \dots, \alpha(v_n))$.
Given conjunctive queries $\mu_1$ and $\mu_2$ over $\tau$, we say that $\mu_1$ is \emph{equivalent} to $\mu_2$ if  $\bA \models \mu_1$ if and only if $\bA \models \mu_2$ for every finite relational $\tau$-structure $\bA$. We say a conjunctive query $\mu$ is \emph{minimal} if every conjunctive query which is equivalent to $\mu$ has at least as many atoms as $\mu$. For every conjunctive query $\mu$, there exists a minimal equivalent query $\mu'$ that can be obtained from $\mu$ by removing zero or
more atoms~\cite{ChandraMerlin}. 

A \emph{multiset relation} on a set $A$ of arity $k$ is a multiset with elements from $A^k$ and a \emph{bag database} $\bA$
over a relational signature $\tau$ consists of a finite domain $A$ and for every $R \in \tau$ of arity $k$, a multiset relation $R^\bA$ of arity $k$. A bag database $\bA$ satisfies a union of conjunctive queries $\mu$ if the relational structure obtained from $\bA$ by forgetting the multiplicities of tuples in its relations satisfies $\mu$.
In the present paper, we study the resilience problem for unions of conjunctive queries in bag semantics; from now on we will refer to this problem just as the \emph{resilience problem}. Let $\tau$ be a finite relational signature and $\mu$ a union of conjunctive queries over $\tau$. The input to the \emph{resilience problem for} $\mu$ consists of 
a bag database $\bA$ over $\tau$, and the 
task is to compute the number of tuples that have to be removed from relations of $\bA$ so that $\bA$ does \emph{not} satisfy $\mu$. This number is called the \emph{resilience} of $\bA$ (with respect to $\mu$). 
As usual, this can be turned into a decision problem where the input also contains a natural number $u \in {\mathbb N}$ and the question is whether the resilience is at most $u$. 
Clearly, $\bA$ does not satisfy $\mu$ if and only if its resilience is $0$. It is easy to see that the resilience problem for any union of conjunctive queries is in NP.

The \emph{canonical database}
of a conjunctive query $\mu$ with relational signature $\tau$ is the relational $\tau$-structure $\bA$ whose domain are the variables of $\mu$ and where $(x_1, \dots, x_k) \in R^{\bA}$ for
$R \in \tau$ of arity $k$ if and only if $\mu$ contains the conjunct $R(x_1, \dots, x_k)$; we denote the canonical database by $\bD_\mu$.

\begin{remark}\label{rem:db}
All terminology introduced for $\tau$-structures also applies to conjunctive queries over $\tau$: by definition, a query has the property if its canonical database has the property. 
\end{remark}
Note that by the above remark, we can talk about homomorphisms between queries and queries being homomorphically incomparable. Observe that if two queries are non-equivalent and minimal, they must be homomorphically incomparable (see, e.g.,~\cite{ChandraMerlin}). 

A relational $\tau$-structure is  \emph{connected} if 
it cannot be written as the disjoint union of two relational $\tau$-structures with non-empty domains.
We show that when classifying the resilience problem for conjunctive queries, it suffices to consider queries that are connected.

\begin{lemma}[\hspace{1sp}{\cite[Lemma 8.5]{Resilience-VCSPs}}]\label{lem:con}
Let $\nu_1,\dots,\nu_k$ be conjunctive queries such that $\nu_i$ does not imply $\nu_j$ if $i \neq j$.
Let $\nu = (\nu_1 \wedge \dots \wedge \nu_k)$ and suppose that $\nu$ occurs in a union $\mu$ of conjunctive queries. For $i \in \{1, \dots, k\}$, let $\mu_i$ be the union of queries obtained by replacing $\nu$ by $\nu_i$ in $\mu$. Then the resilience problem for $\mu$ is NP-hard if the resilience problem for one of the $\mu_i$ is NP-hard.
Conversely, if the resilience problem is in P 
for each $\mu_i$, then the resilience problem for $\mu$ is in P as well.
\end{lemma}

By applying Lemma~\ref{lem:con} finitely many times, we obtain that, when classifying the complexity of the resilience problem for unions of conjunctive
queries, we may restrict our attention to unions of connected conjunctive queries.

\subsection{Connection between resilience and VCSPs}\label{sect:connection}
In this section we summarize the key points of the connection between resilience problems and VCSPs, originally introduced in~\cite{Resilience-VCSPs}.

\begin{definition}\label{def:valued-dual}
Let $\bB$ be a relational $\tau$-structure. Define $\bB_0^1$
to be the valued $\tau$-structure on the same domain as $\bB$ such that
for each $R \in \tau$, $R^{\bB_0^1}(a) = 0$ if $a \in R^{\bB}$ and $R^{\bB_0^1}(a) = 1$ otherwise.
\end{definition}

If $\mu$ is a union of conjunctive queries with signature $\tau$, then a \emph{dual} of $\mu$ is a relational $\tau$-structure $\bB$ with the property that a finite relational $\tau$-structure $\bA$ has a homomorphism to $\bB$ if and only if $\bA$ does not satisfy $\mu$. If $\bB$ and $\bB'$ are both duals of $\mu$, then they are homomorphically equivalent by compactness~\cite[Lemma 4.1.7]{Book}. 

\begin{proposition}[\hspace{1sp}{\cite[Proposition 8.14]{Resilience-VCSPs}}]
\label{prop:connection}
Let $\mu$ be a union of connected conjunctive queries with  signature $\tau$.
Then the resilience problem for $\mu$ 
is polynomial-time equivalent to 
$\VCSP(\bB_0^1)$ for any dual $\bB$ of $\mu$.
\end{proposition}

Let $k \in {\mathbb N}$,  let $C$ be a set and $G$ a permutation group on $C$.  An \emph{orbit of $k$-tuples} of $G$ is a set of the form $\{ \alpha(t) \mid \alpha \in G \}$ for some $t \in C^k$. 
A permutation group $G$ on a~countable set is called \emph{oligomorphic} if for every $k \in {\mathbb N}$ there are finitely many orbits of $k$-tuples in $G$~\cite{Oligo}. 
From now on, whenever we write that a structure has an oligomorphic automorphism group, we also imply that its domain is countable. 
Clearly, every valued structure with a finite domain has an oligomorphic automorphism group. 
A countable relational structure has an oligomorphic automorphism group if and only if it is \emph{$\omega$-categorical}, i.e., if all countable models of its first-order theory are isomorphic~\cite{Hodges}.

A relational $\tau$-structure $\bA$ \emph{embeds} into a relational $\tau$-structure $\bB$ if there is an injective map from $A$ to $B$ that preserves all relations of $\bA$ and their complements; the corresponding map is called an \emph{embedding}. The \emph{age} of a relational $\tau$-structure is the class of all finite relational $\tau$-structures that embed into it.
A relational structure $\bB$ with a relational signature $\tau$ is called 
\begin{itemize}
    \item \emph{finitely bounded} if $\tau$ is finite and there exists a universal $\tau$-sentence $\phi$ such that a~finite relational structure $\bA$
    is in the age of $\bB$ 
iff $\bA \models \phi$;
    \item \emph{homogeneous} if every isomorphism between finite substructures of $\bB$ can be extended to an automorphism of~$\bB$.  
\end{itemize}
If $\bB$ is finitely bounded and homogeneous, then $\Aut(\bB)$ is oligomorphic.

\begin{theorem}[\hspace{1sp}{\cite[Theorem 8.12]{Resilience-VCSPs}}]
\label{thm:css}
For every union $\mu$ of connected conjunctive queries over a finite relational signature $\tau$ there exists a $\tau$-structure $\bB_{\mu}$ such that the following statements hold:
\begin{enumerate}
\item $\bB_\mu$ is a reduct of a finitely bounded and homogeneous structure $\bB$.
\item A countable
$\tau$-structure $\bA$ satisfies $\neg \mu$ if and only if it embeds into $\bB_{\mu}$. 
\item $\bB_\mu$ is finitely bounded.
\item $\Aut(\bB)$ and $\Aut(\bB_{\mu})$ are oligomorphic.  
\end{enumerate}
\end{theorem}

The existence of the dual $\bB_\mu$ for a union of connected conjunctive queries $\mu$ is the key to obtaining another dual $\bC_\mu$, which has a strong model-theoretic property introduced in the following definition.
If $G$ is a permutation group on a set $C$, then $\overline G$ denotes the closure of $G$ in the space $C^C$ with respect to the topology of pointwise convergence. 
This is the unique topology such that the closed subsets of $C^C$ are precisely the endomorphism monoids of relational structures; see, e.g.,~\cite[Proposition 4.4.2]{Book}.
Note that 
$\overline G$ might contain some operations that are not surjective, but if $G = \Aut(\bB)$ for some relational structure $\bB$, then all operations in $\overline G$ are still embeddings of $\bB$ into $\bB$ that preserve all first-order formulas. 

\begin{definition}\label{def:mc-core}
A relational structure
$\bB$ with an oligomorphic automorphism group is a \emph{model-complete core} if $\overline{\Aut(\bB)} = \End(\bB)$.
\end{definition}
For every relational structure $\bB$ with an oligomorphic automorphism group, there exists a model-complete core $\bC$ homomorphically equivalent to $\bB$, which is unique up to isomomorphism called \emph{the model-complete core of $\bB$} \cite[Theorem 16]{Cores-journal}, \cite[Proposition 4.7.7]{Book}. Intuitively, the model-complete core of $\bB$ is in a sense a `minimal' structure with the same CSP as $\bB$. If the domain of $\bB$ is finite, then the domain of its model-complete core (usually just called \emph{core}) is also finite.

The \emph{Gaifman graph} of a 
relational structure $\bA$ 
is the undirected graph with vertex set $A$ where $a,b \in A$ are adjacent if and only if $a \neq b$ and there exists a tuple in a relation of $\bA$ that contains both $a$ and $b$. The Gaifman graph of a conjunctive query is the Gaifman graph of the canonical database of that query.

The following is an analogue of Theorem~\ref{thm:css} for the model-complete core of $\bB_\mu$. 
The statements in the theorem should be considered to be previously known; we provide a proof with references to the literature 
for the convenience of the reader.

\begin{theorem}\label{thm:mcc-dual}
Let $\mu$ be a union of connected conjunctive queries over a finite relational signature $\tau$. Then the model-complete core $\bC_\mu$\footnote{In~\cite{Resilience-VCSPs}, the notation $\bC_\mu$ was used for a different dual of $\mu$, which we do not need in this paper.}
of the structure $\bB_\mu$ from Theorem~\ref{thm:css} satisfies the following:
\begin{enumerate}
\item $\bC_\mu$ is a reduct of a finitely bounded and homogeneous structure $\bB$. 
\item A countable
$\tau$-structure $\bA$ satisfies $\neg \mu$ if and only if there is a homomorphism from $\bA$ to $\bC_{\mu}$.
\item If for each query $\nu$ in $\mu$, the Gaifman graph of $\nu$ is complete, then $\bC_\mu$ is homogeneous.
\item $\Aut(\bB)$ and $\Aut(\bC_{\mu})$ are oligomorphic.  
\end{enumerate}
\end{theorem}
\begin{proof}
Item (1) follows from results in~\cite{MottetPinskerCores}; see~\cite[Corollary 7.5.15]{BodAutomorphismGroups} for an explicit reference.


Item (2) is a consequence of $\bC_{\mu}$ being homomorphically equivalent to $\bB_{\mu}$. 


To prove (3), suppose that for each query $\nu$ in $\mu$, the Gaifman graph of $\nu$ is complete. By \cite[Theorem 8.13]{Resilience-VCSPs}, there exists a dual $\bH$  of $\mu$, which is homogeneous. By \cite[Proposition 4.7.7]{Book}, the model-complete core of $\bH$  is also homogeneous. Note that $\bC_\mu$ is homomorphically equivalent to $\bH$  as they are both duals of $\mu$ and hence, by uniqueness, it is the model-complete core of $\bH$ .

For item (4), note that the automorphism group of $\bB$ is oligomorphic since it is homogeneous with finite relational signature. 
The automorphism group of $\bC_{\mu}$ is oligomorphic, since this property is clearly preserved under taking reducts. 
\end{proof}

Note that since $\bC_\mu$ is unique up to isomorphism and homomorphic equivalence is transitive, the structure $\bC_\mu$ does not depend on the concrete choice of $\bB_\mu$. For a union of connected conjunctive queries $\mu$, let $\Delta_\mu := (\bC_\mu)_0^1$. In most results, this will be the valued structure to which we apply results about $\bB_0^1$ for a dual $\bB$ of $\mu$.

\subsection{Expressive power}\label{sect:expr}
The concept of \emph{expressive power} introduced in this section is a basis for polynomial-time gadget reductions between VCSPs.

\begin{definition}
Let $A$ be a set and $R,R' \in {\mathscr R}_A$.
We say that $R'$ can be obtained from $R$ by 
\begin{itemize}
\item \emph{projecting} if $R'$ is of arity $k$, $R$ is of arity $k+n$ and for all $s \in A^k$ 
\[R'(s) = \inf_{t \in A^n} R(s,t).\]
\item \emph{non-negative scaling} 
 if there exists $a \in {\mathbb Q}_{\geq 0}$ such that $R' = a R$;
 \item  \emph{shifting} if there exists $a \in {\mathbb Q}$ such that $R' = R + a$. 
\end{itemize}
If $R$ is of arity $k$, then the relation that contains all minimal-value tuples of $R$ is 
\[\Opt(R) := \{t \in \Feas(R) \mid R(t) \leq R(s) \text{ for every } s \in A^k\}.\]
\end{definition}

Note that $\inf_{t \in A^n} R(s,t)$ in item (1) might be irrational or $-\infty$. If this is the
case, then $\inf_{t \in A^n} R(s,t)$ does not express a valued relation
because valued relations must have weights from ${\mathbb Q} \cup \{\infty\}$.
However, if $R$ is preserved by all permutations of
an oligomorphic automorphism group, then $R$ 
attains only finitely many values and therefore this is never the case.

If $\mathscr{S} \subseteq \mathscr{R}_A$, then an atomic expression over $\mathscr{S}$ is an atomic $\tau$-expression where $\tau=\mathscr{S}$. We say that $\mathscr{S}$ is \emph{closed under forming sums of atomic expressions} if it contains all valued relations defined by sums of atomic expressions over $\mathscr{S}$.

\begin{definition}[valued relational clone]\label{def:wrelclone}
A \emph{valued relational clone} (over a set $C$) is 
a subset of 
${\mathscr R}_C$ 
that 
is closed under forming sums of atomic expressions,
projecting, shifting, non-negative scaling, $\Feas$, and $\Opt$; we refer to expressions formed this way as \emph{pp-expressions}.
For a valued structure $\Gamma$ with the domain $C$,  
we write 
$\langle \Gamma \rangle$ for the smallest relational clone that contains the valued relations of $\Gamma$. 
If $R \in \langle \Gamma \rangle$, we say that $\Gamma$ \emph{pp-expresses} $R$. 
\end{definition}

The acronym `pp' stands for primitive positive, since the concept of pp-expressions for valued structures is a~generalization of primitive positive definitions used for reductions between CSPs.

\subsection{Fractional maps}
Let $A$ and $B$ be sets.
We equip the space $A^B$ of functions from $B$ to $A$ with the topology of pointwise convergence, where $A$ is taken to be discrete. 
In this topology, a~basis of open sets is given by ${\mathscr S}_{s,t} := \{f \in A^B \mid f(s)=t\}$
for $s \in B^k$ and $t \in A^k$ for some $k \in {\mathbb N}$.
For any topological space $T$, we denote by $\mathcal{B}(T)$
the Borel $\sigma$-algebra on $T$, i.e., the smallest subset of the powerset ${\mathcal P}(T)$ which contains all open sets and is closed under countable intersection and complement. We write $[0,1]$ for the set $\{x \in {\mathbb R} \mid 0 \leq x \leq 1\}$.

\begin{definition}[fractional map]
Let $A$ and $B$ be sets. A~\emph{fractional map} from $B$ to $A$ is a~probability distribution 
$(A^B, \mathcal{B}(A^B),\omega \colon \mathcal{B}(A^B) \to [0,1]),$
that is, $\omega(A^B) = 1$ and $\omega$ is countably additive: if $S_1, S_2,\dots \in \mathcal{B}(A^B)$ are disjoint, then \[\omega\left(\bigcup_{i \in {\mathbb N}} S_i\right) = \sum_{i \in {\mathbb N}} \omega(S_i).\]
\end{definition}

We often use $\omega$ for both the entire fractional map and for the map $\omega \colon \mathcal{B}(A^B) \to [0,1]$.

The set $[0,1]$ carries the topology inherited from  the standard topology on ${\mathbb R}$. We also view ${\mathbb R} \cup \{\infty\}$ as a~topological space with a~basis of open sets given by
all open intervals
$(a,b)$ for $a,b \in {\mathbb R}$, $a<b$ and additionally all sets of the form $\{x \in {\mathbb R} \mid x > a\} \cup \{\infty\}$ (thus, $\mathbb{R} \cup \{ \infty \}$ is equipped with its order topology when ordered in the natural way).

A \emph{(real-valued) random variable} is a~\emph{measurable function} $X \colon T \to {\mathbb R} \cup \{\infty\}$, i.e., pre-images of elements of $\mathcal{B}({\mathbb R} \cup \{\infty\})$ under $X$ are in $\mathcal{B}(T)$. 
If $X$ is a~real-valued random variable, then the \emph{expected value of $X$ (with respect to a~probability distribution $\omega$)} is denoted by $E_\omega[X]$ and is defined 
via the Lebesgue integral 
\[ E_\omega[X] := \int_T X d \omega. \]

In the rest of the paper, 
we will work exclusively with topological spaces $T$ of the form $A^B$ for some sets $A$ and $B$.

\subsection{Pp-constructions}\label{sect:pp-const}
In this section, we introduce a~concept of pp-constructions which generalize pp-expressions and provide polynomial-time reductions between VCSPs. We first define fractional homomorphisms.

\begin{definition}[fractional homomorphism]\label{def:frac-hom}
Let $\Gamma$ and $\Delta$ be valued $\tau$-structures with domains $C$ and $D$, respectively. A~\emph{fractional homomorphism} from
$\Delta$ to $\Gamma$ is a~
fractional map $\omega$ from $D$ to $C$
 such that for every $R \in \tau$ of arity $k$ 
and every tuple $t \in D^k$ 
it holds for the random variable $X \colon C^D \rightarrow \mathbb{R}\cup\{\infty\}$ given  by
$f \mapsto R^{\Gamma}(f(t))$ that 
$E_\omega[X]$
exists and that 
$E_\omega[X]
\leq R^{\Delta}(t).$
\end{definition}

We refer to \cite{Resilience-VCSPs} for a~detailed introduction to fractional homomorphisms.
Two valued $\tau$-structures $\Gamma$ and $\Delta$ are said to be \emph{fractionally homomorphically equivalent}, if there is a fractional homomorphism from $\Gamma$ to $\Delta$ and from $\Delta$ to $\Gamma$.

\begin{remark} \label{rem:fhom-eq}
If $\mu$ is a union of conjunctive queries with duals $\bB$ and $\bC$, then $\bB$ and $\bC$ are homomorphically equivalent. Hence, $\bB_0^1$ and $\bC_0^1$ are fractionally homomorphically equivalent witnessed by fractional maps where the respective homomorphisms have probability $1$.
\end{remark}

As a next step towards the definition of a pp-construction, we define pp-powers.
\begin{definition}[pp-power]
Let $\Gamma$ be a valued structure with a domain $C$ and let $d \in {\mathbb N}$. Then a ($d$-th) \emph{pp-power}
of $\Gamma$ is a valued structure $\Delta$ with the domain 
$C^d$ such that for every valued relation $R$ of $\Delta$ of arity $k$ there exists a valued relation $S$ of arity $kd$ in $\langle \Gamma \rangle$ such that 
$$R((x^1_1,\dots,x^1_d),\dots,(x^k_1,\dots,x^k_d)) = S(x^1_1,\dots,x^1_d,\dots,x^k_1,\dots,x^k_d).$$ 
\end{definition}

We can now define the notion of a pp-construction.
\begin{definition}[pp-construction]
\label{def:pp-constr}
Let $\Gamma, \Delta$ be valued structures. Then  $\Delta$ has a~\emph{pp-construction} in $\Gamma$ if $\Delta$ 
is fractionally homomorphically equivalent to a~structure $\Delta'$ which is a~pp-power of $\Gamma$. 
\end{definition}

The relation of pp-constructability is transitive: if $\Gamma_1$, $\Gamma_2$, and $\Gamma_3$ are valued structures such that $\Gamma_1$ pp-constructs $\Gamma_2$ and $\Gamma_2$ pp-constructs $\Gamma_3$, then $\Gamma_1$ pp-constructs $\Gamma_3$ \cite[Lemma 5.12]{Resilience-VCSPs}.
Note that whenever $\mu$ is a union of connected conjunctive queries and $\Delta_\mu$ pp-constructs a valued structure $\Gamma$, then for every dual $\bB$ of $\mu$, the valued structure $\bB_0^1$ pp-constructs $\Gamma$, because $\Delta_\mu$ and $\bB_0^1$ are fractionally homomorphically equivalent (Remark~\ref{rem:fhom-eq}). 

The motivation for introducing pp-constructions stems from the following lemma: pp-constructions give rise to polynomial-time reductions.
\begin{lemma}[\hspace{1sp}{\cite[Corollary~5.10 and~5.11]{Resilience-VCSPs}}]\label{lem:hard} 
Let $\Gamma$ and $\Delta$ be valued structures 
with finite signatures and oligomorphic automorphism groups such that $\Delta$ has a pp-construction in $\Gamma$. Then there is a~polynomial-time reduction from 
$\VCSP(\Delta)$ to $\VCSP(\Gamma)$. In particular, if $\Delta = (\{0,1\}; \OIT)$, then $\VCSP(\Gamma)$ is NP-hard.
\end{lemma} 

\begin{example}\label{expl:max-cut-hard}
Recall the valued structure $\Gamma_{\textup{MC}}$ from Example~\ref{expl:max-cut-vs}. It is known that $\Gamma_{\textup{MC}}$ pp-constructs $(\{0,1\},\OIT)$~\cite[Example 2.18]{ThesisZaneta} and by Lemma~\ref{lem:hard},
$\VCSP(\Gamma_{\textup{MC}})$ is NP-hard.
\end{example}

\subsection{Fractional polymorphisms}\label{sect:fpol}
We now introduce  \emph{fractional polymorphisms} of valued structures, which generalize polymorphisms of relational structures. 
For valued structures with a~finite domain, our definition specialises to the established notion of a~fractional polymorphism which has been used 
to study the complexity of VCSPs for valued structures over finite domains (see, e.g.~\cite{ThapperZivny13}); it is known that fractional polymorphisms of a finite-domain valued structure capture the complexity of its VCSP up to polynomial-time reductions~\cite{VCSP-Galois, FullaZivny}. Our definition is taken from \cite{Resilience-VCSPs} and allows arbitrary probability distributions in contrast to \cite{ViolaThesis,SchneiderViola,ViolaZivny}.

Let $\ell \in {\mathbb N}$. 
A \emph{fractional operation on $A$ of arity $\ell$} is
a fractional map from $A^\ell$ to $A$. 
The set of all fractional operations on a set $A$ of arity $\ell$ is denoted by ${\mathscr F}^{(\ell)}_A$.

\begin{definition}\label{def:pres}
A fractional operation $\omega \in {\mathscr F}_A^{(\ell)}$ \emph{improves}
a valued relation $R \in {\mathscr R}^{(k)}_A$ if for all $t^1,\dots,t^{\ell} \in A^k$ 
\begin{align}
E := E_\omega[f \mapsto R(f(t^1,\dots,t^{\ell}))]
\text{ exists, and } \quad 
E \leq
\frac{1}{\ell} \sum_{j = 1}^{\ell} R(t^j). \label{eq:fpol}
\end{align}
\end{definition}

Note that~\eqref{eq:fpol} has the interpretation that 
the expected value of $R(f(t^1,\dots,t^\ell))$ is at most the average of the values $R(t^1),\dots$, $R(t^\ell)$. 

\begin{definition}[fractional polymorphism]
If a~fractional operation $\omega$ improves every valued relation in a~valued structure $\Gamma$, then $\omega$ is called a~\emph{fractional polymorphism 
of $\Gamma$}; the set of all fractional polymorphisms of $\Gamma$ is denoted by $\fPol(\Gamma)$.
\end{definition}

\begin{remark}\label{rem:frac-pol-frac-hom}
A fractional polymorphism of arity $\ell$ of a~valued $\tau$-structure $\Gamma$ might also be viewed as a~fractional homomorphism from a~specific $\ell$-th pp-power of $\Gamma$, which we denote by $\Gamma^{\ell}$, to $\Gamma$: 
the domain of $\Gamma^{\ell}$ is $C^{\ell}$,
and for every $R \in \tau$ of arity $k$ we have
\[R^{\Gamma^{\ell}}((x^1_1,\dots,x^1_{\ell}),\dots,(x^k_1,\dots,x^k_{\ell})) 
:= \frac{1}{\ell} \sum_{i=1}^\ell R^{\Gamma}(x^1_i,\dots,x^{k}_i).\]
\end{remark}

\begin{example}\label{expl:id}
Let $A$ be a~set and $\pi^\ell_i \in {\mathscr O}^{(\ell)}_A$ be the $i$-th projection of arity $\ell$, which is given by $\pi^\ell_i(x_1,\dots,x_\ell) = x_i$.
The fractional operation $\Id_\ell$ of arity $\ell$ such that $\Id_\ell(\pi^{\ell}_{i}) = \frac{1}{\ell}$ 
for every $i \in \{1,\dots,\ell\}$ is a~fractional polymorphism of every valued structure with domain $A$. 
\end{example}

\begin{lemma}[Lemma 6.8 in~\cite{Resilience-VCSPs}]\label{lem:easy-Imp-fPol}
Let $\Gamma$ be a~valued $\tau$-structure $\Gamma$ 
over a~countable domain $C$. Then every valued relation
$R \in \langle \Gamma \rangle$ is improved by all fractional polymorphisms of $\Gamma$.
\end{lemma}

Let $\bA$ be a~relational structure and $G$ a permutation group on the domain $A$ of $\bA$. Let $\ell \geq 2$ and $f \colon A^\ell \to A$. The operation $f$ is \emph{pseudo cyclic with respect to $G$} if there exist $e_1, \dots, e_\ell \in \overline G$ such that for all $x_1, \dots, x_\ell \in A$, 
\[e_1f(x_1, x_2, \dots, x_\ell) = e_2 f(x_2, \dots, x_\ell, x_1) = \dots = e_\ell f(x_\ell, x_1, \dots, x_{\ell-1}).\] The operation $f$ is \emph{canonical with respect to $G$} if for all $k \in \N$ and $t^1, \dots, t^\ell \in A^k$, the orbit of the $k$-tuple $f(t^1, \dots, t^\ell)$ with respect to $G$ only depends on the orbits of $t^1, \dots, t^\ell$ with respect to $G$. A~fractional operation $\omega$ on $C$ of arity $\ell$ is called \emph{pseudo cyclic with respect to $G$} if $\omega(S)=1$ for the set $S$ of all pseudo cyclic operations with respect to $G$ of arity $\ell$. \emph{Canonicity} for fractional operations is defined analogously.
The following theorem is a special case of \cite[Theorem 7.13]{Resilience-VCSPs}.

\begin{theorem} \label{thm:tract}
Let $\mu$ be a union of connected conjunctive queries and let $\bA$ be a finitely bounded and homogenous expansion of $\bC_\mu$ (it exists by Theorem~\ref{thm:mcc-dual}). If $\Delta_\mu$ has a canonical pseudo cyclic fractional polymorphism with respect to $\Aut(\bA)$, then $\VCSP(\Delta_\mu)$ and the resilience problem for $\mu$ is in P.
\end{theorem}

We formulate an adaptation of \cite[Conjecture 8.17]{Resilience-VCSPs} for the valued structure $\Delta_\mu$, which replaces the structure $\Gamma_\mu$ used in~\cite{Resilience-VCSPs} (and without considering so-called \emph{exogenous relations}, which we do not introduce in this paper).

\begin{conjecture}\label{conj:res}
Let $\mu$ be a union of connected conjunctive queries over the signature $\tau$ and let $\bA$ be a finitely bounded homogeneous expansion of $\bC_\mu$. Then exactly one of the following holds:
\begin{itemize}
\item $(\{0,1\};\OIT)$
has a pp-construction in $\Delta_\mu$, and $\VCSP(\Delta_\mu)$ is NP-complete.
\item $\Delta_\mu$
has a fractional polymorphism of arity $\ell \geq 2$ which is canonical and pseudo cyclic with respect to $\Aut(\bA)$, and $\VCSP(\Delta_\mu)$ is in P.
\end{itemize}
\end{conjecture}

The main reason to use $\Delta_\mu$ instead of $\Gamma_\mu :=(\bB_\mu)_0^1$ in this conjecture is Corollary~\ref{cor:disjoint}, which shows that for $\Delta_\mu$ the converse of the implication in Conjecture~\ref{conj:res} is true: if $\Delta_\mu$ has a canonical and pseudo cyclic fractional polymorphism, then it does not pp-construct $(\{0,1\}; \OIT)$; see also the discussion in Section~\ref{sect:disjoint}. The relationship between the two conjectures will be a subject of further investigation; at the moment we cannot prove that if $\Delta_\mu$ has a canonical and pseudo cyclic fractional polymorphism, then so does $\Gamma_\mu$, or vice versa.

\section{Disjointness of the two cases of Conjecture~\ref{conj:res}}
\label{sect:disjoint}
In this section we prove that the two cases in the complexity dichotomy of Conjecture~\ref{conj:res} are disjoint.
For a valued structure $\Gamma$, we denote by $\Gamma^*$ the relational structure on the same domain whose relations are all relations from $\langle \Gamma \rangle$ that attain only values $0$ and $\infty$. Observe that by Lemma~\ref{lem:easy-Imp-fPol}, $\Aut(\Gamma) \subseteq \Aut(\Gamma^*)$.

\begin{observation}\label{obs:mc-core}
 Let $\mu$ be a union of conjunctive queries. Then $\Delta_{\mu}^*$ is a model-complete core.
 \end{observation}
\begin{proof}
    Note that for every $R \in \tau$, the structure $\Delta_\mu^*$ contains $R^{\bC_\mu}=\Opt(R^{\Delta_\mu})$. In particular, $\End(\Delta_\mu^*) \subseteq \End(\bC_\mu)$ by Lemma~\ref{lem:easy-Imp-fPol}.
\begin{align*}
\End(\Delta_\mu^*) \subseteq \End(\bC_\mu) & = \overline{\Aut(\bC_\mu)} && \text{($\bC_\mu$ is a model-compl. core}) \\
& = \overline{\Aut(\Delta_\mu)} \subseteq \overline{\Aut(\Delta_\mu^*)} \subseteq \End(\Delta_\mu^*).
\end{align*}
Therefore, $\Delta_\mu^*$ is a model-complete core.
\end{proof}

Let $G$ be a permutation group on a set $C$. An operation $f \colon C^\ell \to C$ on a set $C$ is called \emph{pseudo Taylor with respect to $G$} if for every $i \in \{1, \dots, \ell\}$ there exist $e_1, e_2 \in \overline G$ and variables $z_1, \dots, z_\ell, z_1', \dots, z_\ell' \in \{x,y\}$ such that $z_i \neq z_i'$ and for all $x,y \in C$,
$e_1 (f(z_1, \dots, z_n)) = e_2 (f(z_1', \dots, z_n')).$
A fractional operation $\omega$ on $C$ of arity $\ell$ is called \emph{pseudo Taylor with respect to $G$} if $\omega(T)=1$ for the set $T$ of all pseudo Taylor operations with respect to $G$ on $C$ of arity $\ell$.
Note that every pseudo cyclic operation with respect to $G$ is pseudo Taylor with respect to $G$; similarly, pseudo Taylor fractional operations generalize pseudo cyclic fractional operations.
The following result is not specific to resilience problems, but holds for VCSPs of valued structures with an oligomorphic automorphism group in general.

\begin{theorem}\label{thm:disjoint}
    Let $\Gamma$ be a valued structure with an oligomorphic automorphism group such that
    $\Gamma^*$ is a model-complete core and such that $\Gamma$ has a pseudo cyclic 
    (or, more generally, a pseudo Taylor)
    fractional polymorphism $\omega$ with respect to $\Aut(\Gamma)$. 
    Then 
    $\Gamma$ does not pp-construct $K_3$.
\end{theorem}
\begin{proof}
Suppose for contradiction that $\Gamma$ pp-constructs $K_3$.
By Proposition 2.22 in~\cite{ThesisZaneta}, $\Gamma^*$ pp-constructs $K_3$ as well.
By results in~\cite{BKOPP-equations} (see, e.g., Theorem 10.3.5 in~\cite{Book}), $\Gamma^*$ cannot have a pseudo Taylor
polymorphism with respect to $\Aut(\Gamma)$, and in particular, it cannot have a pseudo cyclic polymorphism with respect to $\Aut(\Gamma)$.

By the definition of a pseudo cyclic fractional operation, 
there is a set $S$ of pseudo cyclic operations of arity $\ell$ on $C$ such that $\omega(S) = 1$.
By Lemma~\ref{lem:easy-Imp-fPol}, $\omega$ is also a fractional polymorphism of $\Gamma^*$.
By Proposition 3.22 in~\cite{ThesisZaneta}, $\omega(S \cap \Pol^{(\ell)}(\Gamma^*)) = 1$. In particular, $S \cap \Pol^{(\ell)}(\Gamma^*)$ is non-empty. This is in contradiction to $\Pol(\Gamma^*)$ not containing any pseudo cyclic operations.
The proof in the case that $\omega$ is just a pseudo Taylor operation is analogous.
\end{proof}

\begin{corollary} \label{cor:disjoint}
Let $\mu$ be a union of conjunctive queries such that $\Delta_\mu$ has a pseudo cyclic, or, more generally, a pseudo Taylor fractional polymorphism $\omega$ with respect to $\Aut(\Delta_\mu)$. 
Then $\Delta_\mu$ does not pp-construct $K_3$.
\end{corollary}
\begin{proof} 
    By Observation~\ref{obs:mc-core}, the structure 
    $\Delta_\mu^*$ is a model-complete core. 
    Now the statement follows from Theorem~\ref{thm:disjoint}. 
\end{proof} 

Observe that to prove Conjecture~\ref{conj:res} it suffices to show that whenever $\Delta_\mu$ does not pp-construct $(\{0,1\}; \OIT)$, it has a canonical pseudo cyclic fractional polymorphism:
this follows from Corollary~\ref{cor:disjoint} (the two cases are known to be disjoint), Theorem~\ref{thm:tract} (the tractability result for canonical pseudo cyclic fractional polymorphisms) and Lemma~\ref{lem:hard} (the hardness condition based on pp-constructions).

The main reason to work with the dual $\bC_\mu$ in this paper, instead of the dual $\bB_\mu$ that was used in~\cite{Resilience-VCSPs}, comes from the proof of Theorem~\ref{thm:disjoint} above: we need the property that $\bC_\mu$ is a model-complete core to get that $\Delta_\mu^*$ is a model-complete core and hence to be able to apply the results from~\cite{BKOPP-equations}.

\section{Complexity Dichotomy for Digraph Resilience Problems}\label{sect:main}
From now on, $R$ denotes a binary relational symbol. We will often view $\{R\}$-structures as directed graphs.
Let
\begin{align*}
    \mu_\ell &:= \exists x \; R(x,x), \\
    \mu_e &:= \exists x,y \; R(x,y), \text{ and}\\
    \mu_{c} &:= \exists x,y \; \big (R(x,y) \wedge R(y,x) \big ).
\end{align*}

The main result of the present article is the following theorem, which is a stronger version of Theorem~\ref{thm:main-informal} presented in Section~\ref{sec:intro}.

\begin{theorem}\label{thm:main-ucq}
Let $\mu$ be a union of conjunctive queries  over the signature $\{R\}$. Then the resilience problem of $\mu$ is in P or NP-complete.
If all conjunctive queries in $\mu$ are minimal, connected, and pairwise non-equivalent, then 
exactly
one of the following holds:
\begin{enumerate}
    \item $\mu$ is equal to $\mu_\ell$, $\mu_e$, or $\mu_c$, and the resilience of $\mu$ is in P. In this case, $\Delta_\mu$ has a fractional polymorphism, which is canonical and pseudo cyclic with respect to $\Aut(\bC_\mu)$.
    \item $\Delta_\mu$ pp-constructs $(\{0,1\}; \OIT)$ and the resilience problem of $\mu$ is NP-complete.
\end{enumerate}
\end{theorem}

 We first sketch the proof strategy for Theorem~\ref{thm:main-ucq}. First observe that one may assume without loss of generality that all queries in $\mu$ are minimal, connected, and pairwise non-equivalent. If $\mu$ is equal to $\mu_\ell$, $\mu_e$, $\mu_c$, then the properties from item 1 are proven in Lemma~\ref{lem:tract}. Otherwise, we prove that either $\mu$ contains a query $\mu_0$ that contains a cycle of length $\geq 3$, or it has a finite dual without directed cycles. In both of these cases we show that item 2 holds.

It is easy to see that the resilience problem for $\mu_\ell$, $\mu_e$ or $\mu_c$ is in P. In Lemma~\ref{lem:tract} we give a stronger algebraic statement 
which corresponds to item 1 in Theorem~\ref{thm:main-ucq};
this was essentially known before, 
but we prove it
for the convenience of the reader. 

\begin{lemma} \label{lem:tract}
For every $\mu \in \{\mu_\ell, \mu_e,\mu_c\}$, the valued structure $\Delta_\mu$ has a canonical pseudo cyclic fractional polymorphism with respect to $\bC_\mu$. In particular, the resilience problem for $\mu$ is in P. 
\end{lemma}
\begin{proof}
Clearly, $\bC_{\mu_e}$ is a structure on $1$-element domain $\{c\}$ with $R^{\bC_{\mu_e}} = \emptyset$. Observe that $\bC_{\mu_e}$ is finitely bounded. Every fractional operation $\omega$ of arity $\geq 2$ is a canonical pseudo cyclic fractional polymorphism of $\Delta_{\mu_e}$ (with respect to $\Aut(\bC_{\mu_e})$). 

It is easy to see that $\bC_{\mu_\ell}$ has a countable domain $C_{\mu_\ell}$ and that $R^{\bC_{\mu_\ell}} = C_{\mu_\ell}^2 \setminus \{(c,c) \mid c \in C_{\mu_\ell}\}$. Note that $\Aut(\bC_{\mu_\ell}) = \Aut(\Delta_{\mu_\ell})$ is the full symmetric group on $C_{\mu_\ell}$. 
Observe that $\bC_{\mu_\ell}$ is finitely bounded, because a finite relational $\{R\}$-structure $\bA$ embeds into $\bC_{\mu_\ell}$ if and only if it satisfies the sentence $\forall x,y \big (\neg R(x,x) \wedge (x =y  \vee R(x,y)) \big)$. 
Let $f \colon C_{\mu_\ell}^2 \to C_{\mu_\ell}$ be injective and let $\omega$ be a binary fractional operation on $C_{\mu_\ell}$ defined by $\omega(f)=1$. It is straightforward to verify that $\omega$ is a fractional polymorphism of $\Delta_{\mu_\ell}$, which is canonical and pseudocyclic with respect to $\Aut(\bC_{\mu_\ell})$.

The only query for which the statement of the lemma is non-trivial is $\mu:= \mu_c$. The following proof is an adaptation of \cite[Example 5.23]{ThesisZaneta} for $\Delta_\mu$.
We show that $\Delta_\mu$ has a ternary canonical pseudo cyclic fractional polymorphism with respect to $\Aut(\bC_\mu)$. 
To increase readability, we write
$R$ for $R^{\bC_\mu}$ and write
$\breve{R}$ for $\{(a,b) \mid (b,a) \in R^{\bC_\mu}\}$. 

A \emph{tournament} is a directed loopless graph such that
between any two distinct vertices $a,b$, the graph contains either the edge $(a,b)$ or the edge $(b,a)$, but not both. Note that $\bC_{\mu}$ must be a tournament: $\mu$ excludes that there are vertices $a,b$ such that both $(a,b)$ and $(b,a)$ is an edge. Suppose for contradiction that there are vertices $a,b$ such that neither $(a,b)$ nor $(b,a)$ forms an edge. Then the graph obtained from 
$\bC_{\mu}$ by adding the edge $(a,b)$ does not satisfy $\mu$, and hence has a homomorphism to $\bC_{\mu}$.
This homomorphism is an endomorphism of $\bC_{\mu}$ which is not an embedding, contradicting the assumption that $\bC_{\mu}$ is a model-complete core.

By Theorem~\ref{thm:mcc-dual}, item (3), the structure $\bC_{\mu}$ is homogeneous, and hence is (isomorphic to) the homogeneous tournament which embeds all finite tournaments. Note that this implies that $\bC_\mu$ is finitely bounded: a finite relational $\{R\}$-structure $\bA$ embeds in to $\bC_\mu$ is and only if it satisfies the sentence \[\forall x,y \big (\neg R(x,x) \wedge \neg(R(x,y) \wedge R(y,x)) \wedge (x=y \vee R(x,y) \vee R(y,x)) \big ).\]
The valued structure $\bC_\mu$ has certain canonical ternary polymorphisms that we will define next; in some sense, they simulate the majority and minority behavior on orbits. 
Consider the tournament $\bT_{\text{majo}}$
whose vertex set is $C_\mu^3$ and which is defined as follows. We put an edge between $(x_0,x_1,x_2)$ and $(y_0,y_1,y_2)$ if 
    \begin{itemize}
        \item for some $i \in \{0,1,2\}$ we have 
        $x_i=y_i$ and $(x_{i+1},y_{i+1})$ forms an edge in $\bC_{\mu}$ (where indices are considered modulo 3), or 
        \item 
        $x_i \neq y_i$ for all $i \in \{0,1,2\}$, and
        $(x_i,y_i)$ is an edge in $\bC_{\mu}$ 
        for at least two distinct arguments $i \in \{1,2,3\}$. 
    \end{itemize}
Since $\bC_\mu$ is a tournament, $\bT_{\text{majo}}$ is a tournament and, in particular, $\bT_{\text{majo}} \models \neg \mu$. Then there exists a homomorphism $f$ from $\bT_{\text{majo}}$ to $\bC_{\mu}$, which is necessarily injective and an embedding.\footnote{The operation $f$ has already been described by Simon Kn\"auer in his Master thesis~\cite[Proposition 5.10]{KnaeuerMaster}. He also describes a minority variant, which is, however, different from the version we describe below (we will also point out what the difference is). The difference is important to later define a fractional polymorphism of $\Delta_{\mu}$.}
By the homogeneity of $\bC_\mu$, the orbits of $k$-tuples of $\Aut(\bC_\mu)$ are determined by the orbits of pairs of entries, and thus it is  clear from the definition that
$f$ is pseudo cyclic and canonical with respect to $\Aut(\bC_\mu)$.

The tournament $\bT_{\text{mino}}$ with vertex set is $C_\mu^3$ is defined analogously: we put an edge between $(x_0,x_1,x_2)$ and $(y_0,y_1,y_2)$ if 
    \begin{itemize}
           \item for some $i \in \{0,1,2\}$ we have 
        $x_i=y_i$ and $(y_{i+1},x_{i+1})$ forms an edge in $\bC_{\mu}$ (where indices are considered modulo 3)\footnote{Here, Kn\"auer in \cite{KnaeuerMaster} required $(x_{i+1},y_{i+1})$ to be an edge, rather than $(y_{i+1},x_{i+1})$.}, or 
        \item 
        $x_i \neq y_i$ for all $i \in \{0,1,2\}$, and
        $(x_i,y_i)$ is an edge in $\bC_{\mu}$ 
        for exactly one argument or exactly three arguments $i \in \{1,2,3\}$. 
    \end{itemize}
Similarly as for $\bT_{\text{majo}}$, we can verify that $\bT_\text{mino} \models \neg \mu$, hence, there exists a homomorphism $g$ from $\bT_{\text{mino}}$ to $\bC_{\mu}$, which is necessarily injective and an embedding. By the same argument as for $f$, the operation $g$ is pseudo cyclic and canonical with respect to $\Aut(\bC_\mu)$.

Let $\omega$ be the ternary fractional operation defined by $\omega(f)=2/3$ and $\omega(g)=1/3$. Note that $\omega$ is a pseudo cyclic and canonical ternary fractional operation on $C_\mu$. We show that $\omega \in \fPol(\Delta_\mu)$. Let $(x,u), (y,v), (z,w) \in C_\mu^2$. We want to verify that
\[E_\omega \left[h \mapsto R^{\Delta_\mu}\left(h \left(\begin{pmatrix}x\\u\end{pmatrix} , \begin{pmatrix}y\\v\end{pmatrix}, \begin{pmatrix}z\\w\end{pmatrix} \right) \right) \right] \leq \frac{1}{3}(R^{\Delta_\mu}(x,u)+R^{\Delta_\mu}(y,v)+R^{\Delta_\mu}(z,w)),\]
equivalently
\begin{align} \label{eq:maj_min}
\begin{split}
2 R^{\Delta_\mu}\left( f\left( \begin{pmatrix}x\\u\end{pmatrix}, \begin{pmatrix}y\\v\end{pmatrix}, \begin{pmatrix}z\\w\end{pmatrix} \right) \right) &+ 
R^{\Delta_\mu}\left(g \left( \begin{pmatrix}x\\u\end{pmatrix}, \begin{pmatrix}y\\v\end{pmatrix}, \begin{pmatrix}z\\w\end{pmatrix} \right)\right) \\
&\leq R^{\Delta_\mu}(x,u)+R^{\Delta_\mu}(y,v)+R^{\Delta_\mu}(z,w).
\end{split}
\end{align}
It is straightforward to verify that no matter how we distribute the tuples 
$(x,u)$, $(y,v)$, and $(z,w)$ between the sets $R$, $\breve{R}$, and $\{(a,a) \mid a \in C_\mu \}$, the inequality~\eqref{eq:maj_min} is satisfied.
We conclude that $\omega \in \fPol(\Delta_\mu)$ and hence $\Delta_\mu$ has a canonical pseudo cyclic fractional polymorphism.

Let $\mu \in \{\mu_\ell, \mu_e, \mu_c\}$. Then the Gaifman graph of $\mu$ is complete and hence, by Theorem~\ref{thm:mcc-dual}, $\bC_\mu$ is homogeneous. In each of the cases, we also showed that $\bC_\mu$ is finitely bounded and that it has a canonical pseudo cyclic fractional polymorphism with respect to $\Aut(\bC_\mu)$. Therefore, by Theorem~\ref{thm:tract}, the resilience problem for $\mu$ is in P.
\end{proof}

\section{Self-join-free queries and self-join variations}
\emph{Self-join-freeness} is a fundamental and frequently used concept in database theory.

\begin{definition}[self-join-free queries]
A union of conjunctive queries $\mu$ is called \emph{self-join-free} if every relation symbol appears at most once in $\mu$.
\end{definition}
Note that this is a more restrictive notion than a union of self-join-free conjunctive queries.

\begin{lemma}\label{lem:sjf-ucq}
Let $\mu$ be a self-join-free union of conjunctive queries over the signature $\tau$ containing a conjunctive query $\nu$ with signature $\rho \subseteq \tau$. 
Let $\bB$ be the $\rho$-reduct of $\bC_\mu$.
Then $\bB$ is a dual of $\nu$, and the $\rho$-reduct of $\Delta_\mu$ is equal to $\bB_0^1$.
\end{lemma}

\begin{proof}
Clearly, $\bB \models \neg \nu$. Let $\bA$ be a finite relational $\rho$-structure such that $\bA \models \neg \nu$. Let $\bA'$ be a $\tau$-expansion of $\bA$ where $R^{\bA'}$ for every $R \in \tau \setminus \rho$ is empty. Then $\bA' \models \neg \mu$ and hence has a homomorphism to $\bC_\mu$. The same map is a homomorphism from $\bA$ to $\bB$. It follows that $\bB$ is a dual of $\nu$.
The last statement
is clear from the definitions.
\end{proof}

We introduce a construction for obtaining queries with self joins from self-join-free queries, which will be crucial in our hardness proofs.

\begin{definition} \label{def:self-join-var}
Let $\nu$ be a self-join-free union of conjunctive queries over the signature $\tau$ and let
$f \colon \tau \to \tau$ be a map that preserves the arities.
Then the union of queries resulting from $\nu$ by replacing each atom $R(x_1, \dots, x_k)$ by $f(R)(x_1, \dots, x_k)$ is denoted by $f(\nu)$.
We say that $f$ is \emph{$\nu$-injective} if for all $R, S \in \tau$ of the same arity $k$ such that $\nu$ contains a query with atoms $R(x_1, \dots, x_k)$ and $S(x_1, \dots, x_k)$ for some variables $x_1, \dots, x_k$, $f(R) \neq f(S)$.
\end{definition}

A union of queries of the form $f(\nu)$ for some self-join-free $\nu$ and arity-preserving $f$ is often called a \emph{self-join variation} of $\nu$ in the literature~\cite{NewResilience}.

\begin{lemma} \label{lem:self-join-var}
Every union of minimal conjunctive queries $\mu$ over a signature $\sigma$ can be written as $f(\nu)$ for some
self-join-free union of conjunctive queries $\nu$ with signature $\tau$ containing $\sigma$ and some $\nu$-injective $f \colon \tau \to \sigma$. 
\end{lemma}

\begin{proof}
For every $R \in \sigma$, let $n_R$ be the number of occurrences of $R$ in $\mu$ and let $R_0 := R$. Let $\tau = \bigcup_{R \in \sigma} \{R_0, R_1, \dots, R_{n_R-1}\}$, where all symbols $R_i$, $i \geq 1$, are fresh and of the same arity as $R$. Define $\nu$ to be the union of conjunctive queries obtained from $\mu$ by replacing the occurrences of $R$ in $\mu$ by $R_0, \dots, R_{n_R-1}$ (each of the symbols is used once) for every $R \in \sigma$; observe that $\nu$ is self-join-free. Let $f \colon \tau \to \sigma \subseteq \tau$ be defined by $f(R_0) = \dots = f(R_{n_R-1}) = R$, $R \in \sigma$. Then $f(\nu)=\mu$. Moreover, $f$ is $\nu$-injective, because queries in $\mu$ are minimal and therefore contain each atom at most once.
\end{proof}

We proceed to present the main result of this section -- Theorem~\ref{thm:self-join-var}. The theorem and its proof is inspired by \cite[Lemma 21]{NewResilience}; their result is a special case of Theorem~\ref{thm:self-join-var}, because it only applies to conjunctive queries rather than unions of conjunctive queries, and because it only states a polynomial-time reduction, whereas our result even provides a pp-construction (which implies a polynomial-time reduction via Lemma~\ref{lem:hard}). 

\begin{theorem}\label{thm:self-join-var}
Let $\nu$ be a self-join-free union of connected conjunctive queries over the signature $\tau$ and let $f \colon \tau \to \tau$ be a $\nu$-injective
map that preserves arities. If all queries in $f(\nu)$ are minimal and pairwise non-equivalent, then $\Delta_{f(\nu)}$ pp-constructs $\Delta_{\nu}$. In particular, the resilience problem for $\nu$ reduces in polynomial time to the resilience problem for $f(\nu)$.
\end{theorem}

\begin{proof}
Let $V$ be the finite set of variables of $\nu$,  which is also the set of variables of $\mu := f(\nu)$ and let $Q$ be the set of conjunctive queries that form the union $\nu$. Note that since all queries in $\nu$ are connected, the same is true for $\mu$. Also, since all queries in $\mu$ are pairwise non-equivalent and minimal, they are pairwise homomorphically incomparable (see Section~\ref{sect:cq}).

Let $D_\mu$ be the domain of $\Delta_\mu$. Let $D := (D_\mu)^{V\times Q}$, i.e., $D$ is a finite power of $D_\mu$; it will be more convenient to use $V \times Q$ as an indexing set rather than the set $\{1, \dots, |V \times Q|\}$. We define a pp-power $\Delta$ of $\Delta_\mu$ on the domain $D$ with the signature $\tau$. For every $R \in \tau$
of arity $k$ and $(d^1, \dots, d^k) \in D^k$, if $R(x_1, \dots, x_k)$ is an atom in a query $\nu_0$ in $\nu$, then
\[R^\Delta(d^1, \dots, d^k):= f(R)^{\Delta_\mu}(d^1_{x_1, \nu_0}, \dots, d^k_{x_k, \nu_0}).\]
The idea is that the combination of  query, relation symbol, and variables uniquely identifies an atom in $\mu$ and therefore encodes the difference between relation symbols $R$ and $R'$ from $\tau$ such that $f(R)=f(R')$. 

Note that since relations in $\Delta_\mu$ are $0$-$1$-valued, the same is true for the relations in $\Delta$ and hence $\Delta = \bB_0^1$ for a relational $\tau$-structure $\bB$ on the domain $D$ where for every $R \in \tau$ and $(d^1, \dots, d^k) \in D^k$, if $R(x_1, \dots, x_k)$ is an atom in a query $\nu_0$ in  $\nu$, then 
\[\bB \models R(d^1, \dots, d^k) \Leftrightarrow \bC_\mu \models f(R)(d^1_{x_1,\nu_0}, \dots, d^k_{x_k,\nu_0})\]
(see Figure~\ref{fig:sjf} for an illustration of the relationship between $\bC_\mu$ and $\bB$).

\begin{figure}
    \centering
    \includegraphics[]{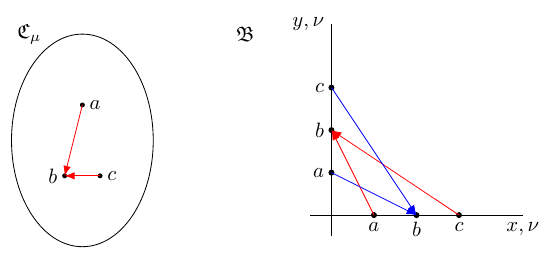}
    \caption{An illustration of the relationship between the structures $\bC_\mu$ and $\bB$ from the proof of Theorem~\ref{thm:self-join-var} for $\nu = \exists x,y (R(x,y) \wedge S(y,x))$ and $\mu = \exists x,y (R(x,y) \wedge R(y,x))$. The tuples in $R$ are depicted in red and tuples in $S$ in blue.}
    \label{fig:sjf}
\end{figure}


\subparagraph*{Claim.} $\bB$ is a dual of $\nu$.

To see this, we first argue that $\bB \not \models \nu$. If $\alpha \colon V \to D$ witnesses that $\bB \models \nu_0$ where $\nu_0$ is a query in $\nu$, then it is straightforward to verify that the map $\alpha' \colon V \to D_\mu$ defined by $v \mapsto \alpha(v)_{v,\nu_0}$ witnesses that $\bC_\mu \models f(\nu_0)$ and hence $\bC_\mu \models \mu$, a contradiction with $\bC_\mu$ being a dual for $\mu$. Therefore, $\bB \not \models \nu$.

It remains to show that if $\bA$ is a relational $\tau$-structure on a finite domain $A$ such that $\bA \not \models \nu$, then $\bA$ maps homomorphically into $\bB$.  To this end, we construct an $f(\tau)$-structure $\bA'$ on the domain $A' = \{a_{v,\nu_0} \mid a \in A, v \in V, \nu_0 \in Q\}$. For every $R \in \tau$, if $R(x_1, \dots, x_k)$ is an atom in a query $\nu_0$ in $\nu$ and $(a_1, \dots, a_k) \in R^\bA$, we put the tuple $((a_1)_{x_1, \nu_0}, \dots, (a_k)_{x_k, \nu_0})$ in $f(R)^{\bA'}$. No other tuples are in the relations of $\bA'$.

 We argue that $\bA' \not \models \mu$. Suppose for contradiction that there exists a a conjunctive query $\mu_0$ in $\mu$ over a variable set $V_0 \subseteq V$ and a map $\beta \colon V_0 \to A'$ witnessing that $\bA' \models \mu_0$. Then for every atom $f(R)(x_1, \dots, x_k)$ in $\mu_0$ we have $(\beta(x_1), \dots, \beta(x_k)) \in f(R)^{\bA'}$. 
We define maps $\beta_Q \colon V_0 \to Q$ and $\beta_V \colon V_0 \to V$
by setting $\beta_Q(x) := \nu_0$ and $\beta_V(x) :=y$ where $\nu_0 \in Q$ and $y \in V$ are such that $\beta(x)=a_{y, \nu_0}$ for some $a \in A$.
Recall that $\mu_0$ is connected. Therefore, by the construction, $\beta_Q$ is constant; let $\nu_0 \in Q$ be the only element of the image of $\beta_Q$. 
 
 Let $f(R)(x_1, \dots, x_k)$ be an atom in $\mu_0$.
Since the tuple $(\beta(x_1), \dots, \beta(x_k))$ has been put in $f(R)^{\bA'}$, $\nu_0$ contains an atom $R'(\beta_V(x_1), \dots, \beta_V(x_k))$ where $R' \in \tau$ is such that $f(R')=f(R)$. Therefore, there is an atom $f(R)(\beta_V(x_1), \dots, \beta_V(x_k))$ in $f(\nu_0)$. Hence, $\beta_V$ defines a homomorphism from $\mu_0$ to $f(\nu_0)$. Since the queries in $\mu$ are pairwise homomorphically incomparable, we must have $f(\nu_0)=\mu_0$. 
Moreover, since $f(R)(x_1, \dots, x_k)$ is an atom in $\mu_0$, we must have the atoms $f(R)(\beta_V(x_1), \dots, \beta_V(x_k))$,
$f(R)(\beta^2_V(x_1), \dots, \beta^2_V(x_k))$, \dots, $f(R)(\beta^p_V(x_1), \dots, \beta^p_V(x_k))$ in $f(\nu_0)=\mu_0$, for all $p \in {\mathbb N}$.

Since $\beta_V$ is a homomorphism from $\mu_0$ to $\mu_0$ and $\mu_0$ is minimal, 
the image of $\beta_V$ is equal to $V_0$. Therefore, $\beta_V \colon V_0 \to V_0$ is surjective.
Since $V_0$ is a finite set, this implies that $\beta_V$ is
a permutation of $V_0$ with an inverse $\beta_V^{-1} = \beta_V^p$ for some $p \in \N$.
By the previous paragraph 
$f(R)(\beta_V^{-1}(x_1), \dots, \beta_V^{-1}(x_k))$ is an atom in $\mu_0$.

Let $\beta'\colon V\to A$ be any map satisfying for every $x \in V_0$ that $\beta'(x) =a$ for the $a \in A$ such that $\beta(\beta_V^{-1}(x))=a_{x,\nu_0}$. Then for every atom $R(x_1, \dots, x_k)$ in $\nu_0$,
we have that $f(R)(x_1, \dots, x_k)$ is an atom of $\mu_0$, and therefore $f(R)(\beta_V^{-1}(x_1), \dots, \beta_V^{-1}(x_k))$ is an atom of $\mu_0$ as well. 
Since $\beta$ witnesses that $\mu_0$ holds in $\bA'$, 
\[(\beta'(x_1)_{x_1,\nu_0}, \dots, \beta'(x_k)_{x_k, \nu_0}) = (\beta(\beta_V^{-1}(x_1)), \dots, \beta(\beta_V^{-1}(x_k))) \in f(R)^{\bA'}.\]
By the $\nu$-injectivity of $f$, there is no atom $R'(x_1, \dots, x_k)$ in $\nu_0$ with $f(R')=f(R)$, so we must have $(\beta'(x_1), \dots, \beta'(x_k)) \in R^\bA$ by the definition of $\bA'$.
Thus, $\beta'$ witnesses that $\bA \models \nu_0$ {and hence, $\bA \models \nu$}, a contradiction. It follows that $\bA' \not \models \mu$.

Since $\bC_\mu$ is a dual of $\mu$, there is a homomorphism $h' \colon \bA' \to \bC_\mu$. 
Let $h \colon \bA \to \bB$ be defined by $h(a) := (h'(a_{v, \nu_0}))_{v\in V, \nu_0 \in Q}$. We claim that $h$ a homomorphism from $\bA$ to $\bB$. To see this, let $R\in \tau$ be of arity $k$ and $(a_1, \dots, a_k) \in R^\bA$. Let $\nu_0$ be a query in $\nu$ with an atom $R(x_1, \dots, x_k)$. Then $((a_1)_{x_1, \nu_0}, \dots, (a_k)_{x_k, \nu_0}) \in f(R)^{\bA'}$ and since $h'$ is a homomorphism, $(h'((a_1)_{x_1,\nu_0}), \dots, h'((a_k)_{x_k, \nu_0})) \in f(R)^{\bC_\mu}$. Then, by the definition of $\bB$,
\[(h(a_1), \dots, h(a_k)) = ((h'((a_1)_{v, \nu_0}))_{v \in V, \nu_0 \in Q}, \dots, (h'((a_k)_{v,\nu_0}))_{v \in V, \nu_0 \in Q}) \in R^{\bB}. \]
It follows that $\bB$ is a dual of $\nu$.

\medskip

Since $\bB$ and $\bC_\nu$ are duals of $\nu$, they are homomorphically equivalent and hence $\Delta$ and $\Delta_\nu$ are fractionally homomorphically equivalent (see Remark~\ref{rem:fhom-eq}). Since $\Delta$ is a pp-power of $\Delta_\mu$, it follows that $\Delta_\mu$ pp-constructs $\Delta_\nu$. The final statement of the theorem follows from Lemma~\ref{lem:hard} and Proposition~\ref{prop:connection}.
\end{proof}

\section{Hardness proofs}
The goal of this section is to present several hardness results that will be used in the proof of Theorem~\ref{thm:main-ucq}.
First we have to define several graph-theoretical notions that will be useful in this section. Let $\bG = (V; E)$ be a directed multigraph and $k \in \N$. A \emph{directed walk} in $\bG$ of \emph{length} $k$ is a sequence $W = (v_0, v_1, \dots, v_k)$ of elements of $V$ such that $(v_i, v_{i+1}) \in E$ for every $i \in \{0, \dots, k-1\}$. The walk $W$ is \emph{closed} if $v_0=v_k$. A \emph{directed path} in $\bG$ of \emph{length} $k$ is a directed walk $(v_0, \dots, v_k)$ such that $v_i \neq v_j$ for all distinct $i,j \in \{0, \dots, k\}$. A \emph{directed cycle} in $\bG$ of length $k$  is a closed directed walk $(v_0, \dots, v_k)$ such that $v_i \neq v_j$ for all distinct $i,j \in \{0, \dots, k-1\}$. An \emph{oriented cycle} in $\bG$ of length $k$ is a sequence $(v_0, v_1, \dots, v_k)$ of elements of $V$ such that $v_k=v_0$, for every $i \in \{0, \dots, k-1\}$, $(v_i, v_{i+1}) \in E$ or $(v_{i+1}, v_i) \in E$, and for every $j \in \{0, \dots, k-1\}$, $j \neq i$, $v_i \neq v_j$.

Suppose now that $\bG$ is undirected. A \emph{cycle in $\bG$} is any sequence that forms an oriented cycle in $\bG$ when viewed as a directed multigraph. We say that $\bG$ is a \emph{tree} if it does not contain any cycles and if it is \emph{connected} in the sense that the graph obtained from $\bG$ by replacing multiple edges by single edges 
is connected (see Section~\ref{sect:cq}).

\subsection{Hardness for queries with orientations of cycles}

A signature $\tau$ is called \emph{binary} if all relation symbols in $\tau$ are binary. 
In this section, we work with binary signatures in general rather than just the signature $\{R\}$.
For any conjunctive query $\mu$ over a binary signature $\tau$, let $\multigr(\mu)$
denote the undirected multigraph whose edge relation is the union (as a multiset) of all the relations of the canonical database $\bD_\mu$.

In this section we prove hardness for the resilience problem for minimal connected conjunctive queries $\mu$ over a binary signature such that $\multigr(\mu)$ contains a cycle of length at least $3$, and, more generally, for unions that contain such a query.
To this end, we start with a result about self-join-free conjunctive queries, which together with Theorem~\ref{thm:self-join-var} will yield a hardness proof for any query over a binary signature. 

\begin{theorem} \label{thm:sjf-cycle}
Let $\mu$ be a connected self-join-free conjunctive query over a binary signature $\tau$.
If $\multigr(\mu)$ contains a cycle of length $\geq 3$, then
$\Delta_\mu$ pp-constructs $(\{0,1\}, \OIT)$.
\end{theorem}

The proof of Theorem~\ref{thm:sjf-cycle} is inspired by a much simpler pp-construction presented in~\cite[Example 8.18]{Resilience-VCSPs} for the query $\mu_{\triangle} := \exists x,y,z (R(x,y) \wedge S(y,z) \wedge T(z,x))$, which is the simplest query in the scope of Theorem~\ref{thm:sjf-cycle}. We recommend having a look at this example as a warm-up for this proof.
The main difference from the general proof is that $\mu_{\triangle}$ is a query with a complete Gaifman graph and therefore has a homogeneous dual\footnote{We remark that the dual used in~\cite[Example 8.18]{Resilience-VCSPs} is a different dual from $\bC_{\mu_{\triangle}}$. 
The dual used there embeds every finite relational $\{R,S,T\}$-structure $\bA$ that does not satisfy $\mu_{\triangle}$, whereas our dual is a model-complete core; so, for example, the empty structure on two vertices maps homomorphically into $\bC_{\mu_\triangle}$, but does not have an embedding.
},
which significantly simplifies the second part of the construction.

\begin{proof}[Proof of Theorem~\ref{thm:sjf-cycle}]
Let $D$ denote the common domain of $\Delta_\mu$ and $\bC_\mu$.
Let $V_\text{all}$ be a countably infinite set of variables. In what follows we will simply write $R$ for $R^{\bC_\mu}$ and $R^{\Delta_\mu}$ where $R \in \tau$; the meaning will always be clear from the context. To increase readability, we will write $R^*$ instead of $\Opt(R)$ in pp-expressions.
We will often refer to the elements of the relations from $\tau$ as edges.
Since $\multigr(\mu)$ contains a cycle of length $\geq 3$, the \emph{directed} multigraph whose edge relation is the union (as a multiset) of all the relations of $\bD_\mu$ contains an oriented cycle $C$ of length $r \geq 3$. 
Since $\mu$ is a self-join-free query, we may assume without loss of generality that $C$ is a directed cycle.
Let $V = \{v_1, \dots, v_n\} \subseteq V_\text{all}$ be the set of variables of $\mu$. 
Let \[\mu = \exists v_1, \dots, v_n (R(x,y) \wedge S(y,z) \wedge T(z,w) \wedge \mu_C  \wedge \mu'),\] where $R$, $S$, $T$ are distinct binary relation symbols and $x$, $y$, $z$ are distinct variables from $\{v_1, \dots, v_n\}$. 
If $r = 3$, then $x=w$ and $\mu_{C}$ is the empty conjunction. 
If $r \geq 4$, then $w \notin \{x,y,z\}$ and $R(x,y) \wedge S(y,z) \wedge T(z,w) \wedge \mu_C$ forms the directed cycle $C$. 
Note that $\mu'$ is a conjunction of atoms over variables from $V$ that does not contain any of the relation symbols from $\{R,S,T\}$ or $\mu_C$.

We will work in the situation where $r \geq 4$; the proof can be adapted straightforwardly to the case $r=3$ where $z=w$. 
We may assume without loss of generality that $(v_1, v_2, v_3, v_4) = (x,y,z,w)$. Since $\mu_C$ cannot contain $v_2$ and $v_3$ and necessarily contains $v_1$ and $v_4$, we may assume that $\mu_C = \mu_C(v_1, v_4, v_5, \dots, v_n)$. Let $\tau_C \subsetneq \tau$ be the signature of $\mu_C$.
The cycle $C$ is illustrated in Figure~\ref{fig:mu}. The relations $R$, $S$, $T$ are denoted by red, blue, and green edge,
respectively. Black edges represent the rest of the cycle composed from edges from $\tau_C$. The atoms with symbols from $\tau \setminus (\{R,S,T\} \cup \tau_C)$ are not depicted.

\begin{figure}
    \centering
    \includegraphics[]{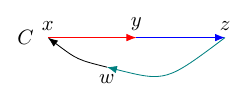}
    \caption{An illustration of the cycle $C$.}
    \label{fig:mu}
\end{figure}

Let $\mu^*(v_1, \dots, v_n)$ denote the formula $\mu_C \wedge \mu'$.
Let $\psi^*$ denote the pp-expression obtained from $\mu^*$ by replacing all occurrences of $\wedge$ by $+$ and all symbols $R$ by $R^*$. We also need the following notation. If $k,\ell \in \N$, $k \leq \ell$, and $u_1, \dots, u_k \in V_\text{all}$,
then $(u_1, \dots, u_k)^+$ denotes a tuple from $V_\text{all}^\ell$ which has on every  position $i\in \{1, \dots, k\}$ the variable $u_i$ and fresh variables on all remaining positions, i.e., variables that have not appeared before. The length $\ell$ of the tuple will always be clear from the context.
If $\phi(x_1, \dots, x_\ell)$ is a first-order $\tau$-formula, we write $\phi^+(x_1, \dots, x_k)$ for the first order formula $\exists x_{k+1}, \dots, x_\ell. \; \phi(x_1, \dots, x_k, x_{k+1}, \dots, x_\ell)$.
In what follows, we use the formulation that an atomic expression \emph{holds} if it evaluates to $0$, and an atomic expression \emph{is violated} if it does not hold. Since we will evaluate pp-expressions only in $\Delta_\mu$, we will write $\psi$ instead of $\psi^{\Delta_\mu}$ for every pp-expression $\psi$ over $\tau$. The cost of $\psi$
always means the cost of $\psi$ in $\Delta_\mu$. All the infima in the pp-expressions below are taken over the difference of the set of variables that appear in the pp-expression and the free variables of the respective pp-expression.

Next, we define auxiliary pp-expressions $\psi_R$, $\psi_S$, and $\psi_T$ that will be used in the construction.
The pp-expression $\psi_R$ 
is defined by
\begin{align*}
\psi_R(x_0,y_0,x_1,y_1) = \inf \big( &R(x_0,y_0) + S(y_0,z_0) + T^*(z_0,w_0) + \psi^*((x_0,y_0,z_0,w_0)^+)  \\
+ &R^*(x_1,y_0) + S(y_0,z_0) + T(z_0,w_1) + \psi^*((x_1,y_0,z_0,w_1)^+)\\
+ &R(x_1,y_1) + S^*(y_1,z_0) + T(z_0,w_1) + \psi^*((x_1,y_1,z_0,w_1)^+)\big).
\end{align*}
Note that the cost of $\psi_R$ is at least $3$, since each row of the expression 
contains a copy of $\mu$ with two atoms that are not crisp.

Below we define quantifier-free formulas
$\phi_R^0$ and $\phi_R^1$
with the property that for all $x_0, y_0, x_1, y_1 \in D$, 
\[\bC_\mu \models \Opt(\psi_R)(x_0,y_0,x_1,y_1) \Leftrightarrow (\phi_R^0 \vee \phi_R^1)^+(x_0,y_0,x_1,y_1);\]
see Figure~\ref{fig:R} for an illustration of $\psi_R$, $\phi_R^0$, and $\phi_R^1$. 
To avoid listing all the variables, we assume that the variables $x_0, y_0, x_1, y_1$ are the first four entries of $\phi_R^0$ and $\phi_R^1$, and we denote the remaining variables by `$\dots$'. Since $\phi_R^0$ and $\phi_R^1$ are first-order formulas, repetitions of the same atom do not make a difference. However, we leave a blank space in the place where an atom from the previous line could be repeated for better comparison with $\psi_R$. Let
\[
\begin{alignedat}{4}
&\phi_R^0(x_0,y_0,x_1,y_1, \dots)
& \; := \neg R(x_0,y_0) &\quad &\wedge S(y_0,z_0) &\quad &\wedge T(z_0,w_0) &\wedge \mu^*((x_0,y_0,z_0,w_0)^+)  \\
&             &\wedge R(x_1,y_0) &\quad &\phantom{\wedge S(y_0,z_0)} &\quad &\wedge \neg T(z_0,w_1) &\wedge \mu^*((x_1,y_0,z_0,w_1)^+)\\
&             &\wedge R(x_1,y_1) &\quad &\wedge S(y_1,z_0) &\quad &\phantom{\wedge T(z_0,w_1)} &\wedge \mu^*((x_1,y_1,z_0,w_1)^+)
\end{alignedat}
\]
and
\[
\begin{alignedat}{4}
&\phi_R^1(x_0,y_0,x_1,y_1, \dots)
& \; := R(x_0,y_0) &\wedge \neg S(y_0,z_0) &\wedge T(z_0,w_0) &\wedge \mu^*((x_0,y_0,z_0,w_0)^+)  \\
&             & \wedge R(x_1,y_0) &\phantom{\wedge \neg S(y_0,z_0)} & \wedge T(z_0,w_1) &\wedge \mu^*((x_1,y_0,z_0,w_1)^+)\\
&             &\wedge \neg R(x_1,y_1) &\wedge S(y_1,z_0) &\phantom{\wedge T(z_0,w_1)} &\wedge \mu^*((x_1,y_1,z_0,w_1)^+).
\end{alignedat}
\]

We argue that $\phi_R^0$ and $\phi_R^1$ are satisfiable in $\bC_\mu$. Let $(\phi_R^0)'$ be the primitive positive sentence resulting from $\phi_R^0$ by omitting the negated atoms and existentially quantifying all its variables. The canonical database of $(\phi_R^0)'$ does not satisfy $\mu$ and therefore maps homomorphically into $\bC_\mu$. Since $\bC_\mu \not \models \mu$, the image of the canonical database shows 
that $\phi_R^0$ is satisfiable in $\bC_\mu$; the argument for $\phi_R^1$ is analogous.
Note that every tuple that satisfies $\phi_R^0$ or $\phi_R^1$ in $\bC_\mu$ certifies that the cost $3$ of $\psi_R$ can be achieved. Therefore, $\Opt(\psi_R)$ consists precisely of tuples that realize the cost $3$. 
Also note that from each copy of $\mu$ in $\psi_R$, one of the two edges that are not crisp needs to be violated to realize the infimum from the definition of $\Opt(\psi_R)$. To realize the cost $3$, there are only two options: either to violate $R(x_0, y_0)$ and $T(z_0,w_1)$, or to violate $S(y_0,z_0)$ and $R(x_1, y_1)$. Therefore, for all $x_0,y_0,x_1,y_1 \in D$ we have
\begin{align*}
\bC_\mu &\models \Opt(\psi_R)(x_0,y_0,x_1,y_1) \Leftrightarrow (\phi_R^0 \vee \phi_R^1)^+(x_0,y_0,x_1,y_1) 
\end{align*}
and, in particular, 
\begin{align}\label{eq:opt-psiR}
    \bC_\mu &\models \Opt(\psi_R)(x_0,y_0,x_1,y_1) \Rightarrow (\neg R(x_0,y_0) \wedge R(x_1, y_1)) \vee (R(x_0,y_0) \wedge \neg R(x_1, y_1)). 
\end{align}
In this sense, $\Opt(\psi_R)$ implies an  alternation of two $R$-edges. Similarly, the pp-expression $\Opt(\psi_S)$ will imply an alternation of an $R$-edge and an $S$-edge, and $\Opt(\psi_T)$ will imply an alternation of an $R$-edge and a $T$-edge.

\begin{figure}
    \centering
    \includegraphics[]{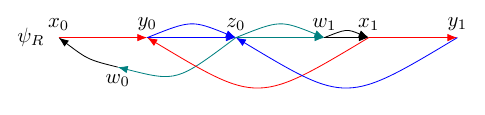}
    \includegraphics[]{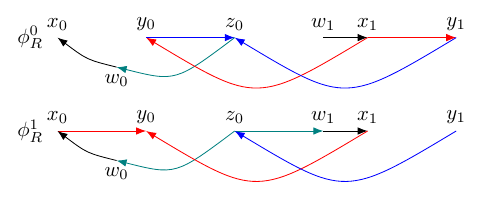}
    \caption{An illustration of the pp-expression $\psi_R$ and the quantifier-free formulas $\phi_R^0$ and $\phi_R^1$.}
    \label{fig:R}
\end{figure}

The pp-expression $\psi_S$ is defined by
\begin{align*}
\psi_S(x_0,y_0,y_2,z_1) := \inf \big(
& R(x_0,y_0) + S^*(y_0,z_0) + T(z_0,w_0) + \psi^*((x_0,y_0,z_0,w_0)^+)  \\
+ &R^*(x_0,y_1) + S(y_1,z_0) + T(z_0,w_0)+ \psi^*((x_0,y_1,z_0,w_0)^+)\\
+ &R(x_1,y_1) + S(y_1,z_0) + T^*(z_0,w_1) + \psi^*((x_1,y_1,z_0,w_1)^+) \\
+ &R(x_1,y_1) + S^*(y_1,z_1) + T(z_1,w_2) + \psi^*((x_1,y_1,z_1,w_2)^+) \\
+ &R^*(x_1,y_2) + S(y_2,z_1) + T(z_1,w_2) + \psi^*((x_1,y_2,z_1,w_2)^+)
\big).
\end{align*}

Note that the cost of $\psi_S$ is at least $5$, because each row of the expression contains a copy of $\mu$ with two summands that are not crisp. 

Below we define quantifier-free formulas $\phi_S^0$ and $\phi_S^1$ 
with the property that for all $x_0,y_0,y_2,z_1 \in D$
\[\bC_\mu \models \Opt(\psi_S)(x_0,y_0,y_2,z_1) \Leftrightarrow (\phi_S^0 \vee \phi_S^1)^+(x_0,y_0,y_2,z_1);\]
see Figure~\ref{fig:S} for the illustration of $\psi_S$, $\phi_S^0$, and $\phi_S^1$.
Again, to avoid listing all the variables, we assume that the variables $x_0, y_0, y_2, z_1$ are the first four entries of $\phi_S^0$ and $\phi_S^1$ and indicate the remaining variables by `$\dots$'. Let
\[
\begin{alignedat}{4}
&\phi_S^0(x_0,y_0,y_2,z_1, \dots) \; & := \neg R(x_0,y_0) &\wedge S(y_0,z_0) &\wedge T(z_0,w_0) &\wedge \psi^*((x_0,y_0,z_0,w_0)^+)  \\
&                          &\wedge R(x_0, y_1) &\wedge \neg S(y_1,z_0) &\phantom{\wedge T(z_0, w_0)} &\wedge \psi^*((x_0,y_1,z_0,w_0)^+)\\
&                          &\wedge R(x_1, y_1) &\phantom{\wedge \neg S(y_1,z_0)} &\wedge T(z_0,w_1) &\wedge \psi^*((x_1,y_1,z_0,w_1)^+) \\
&                          &\phantom{\wedge R(x_1, y_1)} &\wedge S(y_1, z_1) &\wedge \neg T(z_1,w_2) &\wedge \psi^*((x_1,y_1,z_1,w_2)^+) \\
&                          &\wedge R(x_1,y_2) &\wedge S(y_2,z_1) &\phantom{\wedge \neg T(z_1,w_2)} &\wedge \psi^*((x_1,y_2,z_1,w_2)^+)
\end{alignedat}
\]
and
\[
\begin{alignedat}{4}
&\phi_S^1(x_0,y_0,y_2,z_1, \dots) \; 
& := R(x_0,y_0)  &\wedge S(y_0,z_0) &\wedge \neg T(z_0,w_0)   &\wedge \psi^*((x_0,y_0,z_0,w_0)^+)  \\
&               &\wedge R(x_0, y_1)  &\wedge S(y_1,z_0)  &\phantom{\wedge \neg T(z_0, w_0)} &\wedge \psi^*((x_0,y_1,z_0,w_0)^+)\\
&               &\wedge \neg R(x_1, y_1) &\phantom{\wedge S(y_1,z_0)} &\wedge T(z_0,w_1)         &\wedge \psi^*((x_1,y_1,z_0,w_1)^+) \\
&        &\phantom{\wedge \neg R(x_1, y_1)} &\wedge S(y_1, z_1) &\wedge T(z_1,w_2)            &\wedge \psi^*((x_1,y_1,z_1,w_2)^+) \\
&         &\wedge R(x_1,y_2)   &\wedge \neg S(y_2,z_1)   &\phantom{\wedge T(z_1,w_2)}     &\wedge \psi^*((x_1,y_2,z_1,w_2)^+).
\end{alignedat}
\]
Note that each of $\phi_S^0$ and $\phi_S^1$ is satisfiable in $\bC_\mu$, because $\bC_\mu$ is a dual of $\mu$. Moreover, every tuple that satisfies $\phi_S^0$ or $\phi_S^1$ in $\bC_\mu$ certifies that the cost $5$ of $\psi_S$ can  be achieved. Therefore, $\Opt(\psi_S)$ consists of precisely those tuples that realize the cost $5$. 
Since from each copy of $\mu$ in $\psi_S$ one of the two edges that are not crisp needs to be violated, one can verify analogously as for $\psi_R$ that
for all $x_0,y_0,y_2,z_1 \in D$ we have
\begin{align*}
\bC_\mu &\models \Opt(\psi_S)(x_0,y_0,y_2,z_1) \Leftrightarrow (\phi_S^0 \vee \phi_S^1)^+(x_0,y_0,y_2,z_1) 
\end{align*}
and, in particular, 
\begin{align}\label{eq:opt-psiS}
    \bC_\mu &\models \Opt(\psi_S)(x_0,y_0,y_2,z_1) \Rightarrow (\neg R(x_0,y_0) \wedge S(y_2, z_1)) \vee (R(x_0,y_0) \wedge \neg S(y_2, z_1)) . 
\end{align}
In this sense, $\Opt(\psi_S)$ implies an alternation of an $R$-edge and an $S$-edge.

\begin{figure}
    \centering
    \includegraphics[]{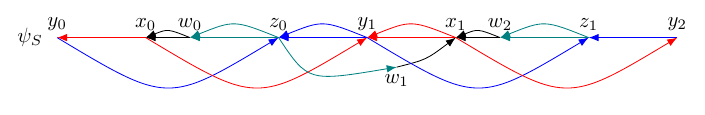}
    \includegraphics[]{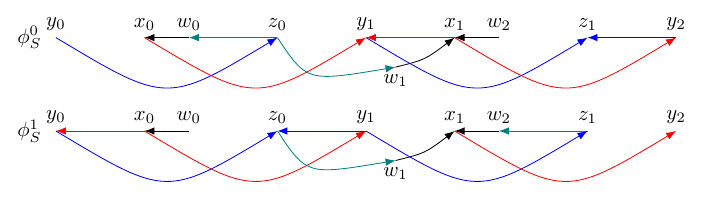}
    \caption{An illustration of the pp-expression $\psi_S$ and the quantifier-free formulas $\phi_S^0$ and $\phi_S^1$.}
    \label{fig:S}
\end{figure}

The pp-expression $\psi_T$ is defined as follows. 
\begin{align*}
\psi_T(x_0, y_0, z_1, w_3) := \inf \big( 
&R(x_0, y_0) + S(y_0,z_0) + T^*(z_0,w_0) + \psi^*((x_0, y_0, z_0, w_0)^+)  \\
+ &R^*(x_1,y_0) + S(y_0, z_0) + T(z_0, w_1)+ \psi^*((x_1, y_0, z_0, w_1)^+)\\
+ &R(x_1, y_1) + S^*(y_1,z_0) + T(z_0,w_1) + \psi^*((x_1,y_1,z_0,w_1)^+) \\
+ &R(x_1,y_1) + S(y_1,z_1) + T^*(z_1,w_2) + \psi^*((x_1,y_1,z_1,w_2)^+) \\
+ &R^*(x_2,y_1) + S(y_1,z_1) + T(z_1,w_3) + \psi^*((x_2,y_1,z_1,w_3)^+)
\big)
\end{align*}
Analogously to $\psi_S$, the cost of $\psi_T$ is at least $5$. Below we define the quantifier-free formulas $\phi_T^0$ and $\phi_T^1$ 
with the property that for all $x_0, y_0, z_1, w_3 \in D$ 
\[\bC_\mu \models \Opt(\psi_T)(x_0, y_0, z_1, w_3) \Leftrightarrow (\phi_T^0 \vee \phi_T^1)^+(x_0, y_0, z_1, w_3);\]
see Figure~\ref{fig:T} for an  illustration of $\psi_T$, $\phi_T^0$ and $\phi_T^1$. 
Again, to avoid listing all the variables, we assume that the variables $x_0, y_0, z_1, w_3$ are the first four entries of $\phi_T^0$ and $\phi_T^1$ and denote the remaining variables by `$\dots$'. Let
\[
\begin{alignedat}{4}
&\phi_T^0(x_0, y_0, z_1, w_3, \dots) 
&=  \neg R(x_0,y_0) &\wedge S(y_0,z_0) &\wedge T(z_0,w_0) &\wedge \mu^*((x_0, y_0, z_0, w_0)^+)  \\
& &\wedge R(x_1,y_0) &\phantom{\wedge S(y_0, z_0)} &\wedge \neg T(z_0, w_1) &\wedge \mu^*((x_1, y_0, z_0, w_1)^+)\\
& &\wedge R(x_1, y_1) &\wedge S(y_1,z_0) &\phantom{\wedge \neg T(z_0,w_1)} &\wedge \mu^*((x_1,y_1,z_0,w_1)^+) \\
& &\phantom{\wedge R(x_1, y_1)} &\wedge \neg  S(y_1,z_1) &\wedge T(z_1,w_2) &\wedge \mu^*((x_1,y_1,z_1,w_2)^+) \\
& &\wedge R(x_2,y_1) &\phantom{\wedge \neg S(y_1,z_1)} &\wedge T(z_1,w_3) &\wedge \mu^*((x_2,y_1,z_1,w_3)^+)
\end{alignedat}
\]
and
\[
\begin{alignedat}{4}
&\phi_T^1(x_0, y_0, z_1, w_3, \dots)
&= R(x_0, y_0) &\wedge \neg S(y_0,z_0) &\wedge T(z_0,w_0) &\wedge \mu^*((x_0, y_0, z_0, w_0)^+)  \\
& &\wedge R(x_1,y_0) &\phantom{\wedge \neg S(y_0,z_0)} &\wedge T(z_0, w_1) &\wedge \mu^*((x_1, y_0, z_0, w_1)^+)\\
& &\wedge  \neg R(x_1, y_1) &\wedge S(y_1,z_0) &\phantom{\wedge T(z_0,w_1)} &\wedge \mu^*((x_1,y_1,z_0,w_1)^+) \\
& &\phantom{\wedge \neg R(x_1, y_1)} &\wedge S(y_1,z_1) &\wedge T(z_1,w_2) &\wedge \mu^*((x_1,y_1,z_1,w_2)^+) \\
& &\wedge R(x_2,y_1) &\phantom{\wedge S(y_1,z_1)} &\wedge \neg T(z_1,w_3) &\wedge \mu^*((x_2,y_1,z_1,w_3)^+).
\end{alignedat}
\]
Each of $\phi_T^0$ and $\phi_T^1$ is satisfiable in $\bC_\mu$, because 
$\bC_\mu$ is a dual of $\mu$. 
Moreover, every tuple that satisfies $\phi_T^0$ or $\phi_T^1$ in $\bC_\mu$ certifies that the the cost $5$ of $\psi_T$ can  be achieved. Therefore, $\Opt(\psi_T)$ consists of precisely those tuples that realize the cost $5$. Similarly to $\psi_R$ and $\psi_S$, we obtain that
for all $x_0,y_0,z_1,w_3 \in D$
we have
\begin{align*}
\bC_\mu &\models \Opt(\psi_T)(x_0,y_0,z_1,w_3) \Leftrightarrow (\phi_T^0 \vee \phi_T^1)^+(x_0,y_0,z_1,w_3) 
\end{align*}
and, in particular, 
\begin{align}\label{eq:opt-psiT}
    \bC_\mu &\models \Opt(\psi_T)(x_0,y_0,z_1,w_3) \Rightarrow (\neg R(x_0,y_0) \wedge T(z_1, w_3)) \vee (R(x_0,y_0) \wedge \neg T(z_1, w_3)). 
\end{align}
In this sense, $\Opt(\psi_T)$ implies an alternation of an $R$-edge and a $T$-edge.

\begin{figure}
    \centering
    \includegraphics[]{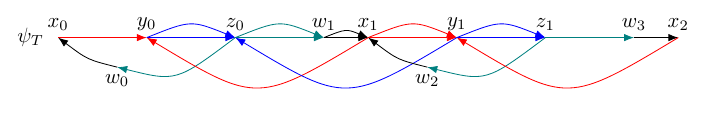}
    \includegraphics[]{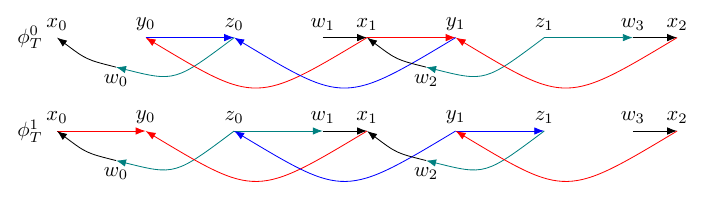}
    \caption{An illustration of the pp-expression $\psi_T$ and the quantifier-free formulas $\phi_T^0$ and $\phi_T^1$.}
    \label{fig:T}
\end{figure}

Let $\psi(x_R, y_R, x_S, y_S, x_T, y_T)$ be the pp-expression
\begin{align} 
\inf \big( &R(x,y) + S(y,z) + T(z,w) + \psi^*((x,y,z,w)^+) \label{eq:psi1}\\
+ &\Opt(\psi_R(x_R, y_R, x,y)) + \Opt(\psi_S(x_S, y_S ,y,z)) + \Opt(\psi_T(x_T, y_T,z,w))
\label{eq:psi2}
\big).
\end{align}
Note that the expression in \eqref{eq:psi2} attains only values $0$ and $\infty$. Clearly, the cost of $\psi$ is at least $1$, because \eqref{eq:psi1} contains a copy of $\mu$ and for every tuple that realizes the cost $1$ precisely one of the atoms $R(x,y)$, $S(y,z)$, and $T(z,w)$ is violated. By \eqref{eq:opt-psiR}, \eqref{eq:opt-psiS} and \eqref{eq:opt-psiT}, this implies that precisely one of the atoms $R(x_R, y_R)$, $R(x_S, y_S)$ and $R(x_T, y_T)$ holds. Intuitively, this simulates the behavior of the boolean relation $\OIT$ that we want to pp-construct.

We claim there exist $a',b',c',d' \in D$ such that $(a',b') \in R$, $(c',d') \notin R$ and \[(a',b',c',d',c',d'), (c',d',a',b',c',d'), (c',d',c',d',a',b') \in \Opt(\psi).\]
These elements will be needed to obtain a homomorphism from $(\{0,1\};\OIT)$ to a $2$-dimensional pp-power of $\Delta_\mu$.
We obtain these elements by constructing a finite relational $\tau$-structure $\bA$ which maps homomorphically into $\bC_\mu$ and whose domain $A$ contains elements $a,b,c,d$; the elements $a',b',c',d'$ will be the images of $a,b,c,d$ under this homomorphism, respectively.

First note that for any structure $\bA$ with a domain $A \subseteq V_\text{all}$, if for some first-order formula $\phi(x_1, \dots, x_\ell)$, $k \leq \ell$, $u_1,\dots,u_k \in A$, and some extension of the tuple $(u_1, \dots, u_k)$ to an $\ell$-tuple $(u_1, \dots, u_k)^+$ we have that $\bA \models \phi((u_1, \dots, u_k)^+)$, then $\bA \models \phi^+(u_1, \dots, u_k)$.
Let $a,b,c,d$ and $x^Q, y^Q, z^Q, w^Q$, for $Q \in \{R,S,T\}$, be distinct elements of $V_\text{all}$. 
Let $A$ be a finite subset of $V_\text{all}$ consisting of all elements that appear in the definition of the relations of $\bA$ below and elements completing the tuples of the form $(u_1, \dots, u_k)^+$ in that definition.
Let $\bA$ be the $\tau$-structure with domain $A$ whose relations
are defined by the following conditions.
\begin{enumerate}[(i)]
    \item $\bA \models S(y^R,z^R) \wedge T(z^R,w^R) \wedge \mu^*((x^R, y^R, z^R, w^R)^+)$;
    \item $\bA \models \phi_R^1((a,b,x^R, y^R)^+)$;
    \item $\bA \models \phi_S^0((c,d, y^R, z^R)^+)$;
    \item $\bA \models \phi_T^0((c,d,z^R, w^R)^+)$;
    \item $\bA \models R(x^S,y^S) \wedge T(z^S, w^S) \wedge \mu^*((x^S, y^S,z^S,w^S)^+)$;
    \item $\bA \models \phi_R^0((c,d,x^S,y^S)^+)$;
    \item $\bA \models \phi_S^1((a,b,y^S, z^S)^+)$;
    \item $\bA \models \phi_T^0((c,d,z^S,w^S)^+)$;
    \item $\bA \models R(x^T, y^T) \wedge S(y^T,z^T) \wedge \mu^*((x^T,y^T,z^T,w^T)^+)$;
    \item $\bA \models \phi_R^0((c,d,x^T,y^T)^+)$;
    \item $\bA \models \phi_S^0((c,d,y^T, z^T)^+)$;
    \item $\bA \models \phi_T^1((a,b,z^T,w^T)^+)$;
    \item no more tuples than those forced by the conditions above lie in the relations of $\bA$.
\end{enumerate}
Note that the conditions are compatible with each other and hence the structure $\bA$ exists; for example, item (ii) implies that $\bA \models R(a,b) \wedge \neg R(x^R,y^R)$, which is compatible with the edge $R(x^R, y^R)$ being left out in (i).

We provide some intuition about $\bA$. Consider the substructure induced on the elements of $\bA$ appearing in items (i)-(iv); for a reason that will become apparent in a moment we refer to it as the \emph{$R$-copy}. Item (i) together with (xiii) implies that we leave out an $R$-edge from a copy from $\mu$,
exhibiting one of the optimal behaviors for the expression in \eqref{eq:psi1} when interpreted in $\bC_\mu$.  Item (ii) implies that $\bA \models R(a,b) \wedge \neg R(x^R,y^R)$ and that $(a,b,x^R,y^R)$ would lie in $\Opt(\psi_R)$ if the $R$-copy were a substructure of $\bC_\mu$. Similarly, item (iii) implies that $\bA \models \neg R(c,d) \wedge S(y^R,z^R)$ and that $(c,d,y^R, z^R)$ would lie in $\Opt(\psi_S)$ if the $R$-copy were a substructure of $\bC_\mu$. Finally, item (iv) implies that $\bA \models \neg R(c,d) \wedge T(z^R,w^R)$ and $(c,d,z^R,w^R)$ would lie in $\Opt(\psi_T)$ if the  $R$-copy were a substructure of $\bC_\mu$. Analogously, we will call the substructure stemming from (v)--(viii) the \emph{$S$-copy} and from (ix)--(xii) the \emph{$T$-copy}; their properties are analogous to the properties of the $R$-copy with the obvious alterations.

Let $M$ denote the directed multigraph on the domain $A$ whose edge relation is the union (as a multiset) of $R^\bA$, $S^\bA$, $T^\bA$, and $Q^\bA$ for $Q \in \tau_C$; relation symbols from $\tau \setminus (\{R,S,T\} \cup \tau_C)$ are not included.

\subparagraph*{Claim 1.} $\bA \models \neg \mu$.

In fact, we prove a stronger statement, namely that $M$ does not contain a closed directed walk of positive length, which implies that $\bA$ does not contain 
a copy of the cycle $C$. Note that by the construction, the $R$-copy, the $S$-copy, and the $T$-copy satisfy 
\[\neg \big( \exists x,y,z,w, \dots (R(x,y) \wedge S(y,z) \wedge T(z,w) \wedge \mu_C((x,w)^+))),\]
where `$\dots$' stands for the variables introduced to complete the tuple $(x,w)^+$.
Therefore, if $M$ contains a closed directed walk, it includes vertices of at least two of these copies. Note that the only vertices of $M$ that appear in more than one copy of $\psi$ are $a$, $b$, $c$, and $d$, so any closed directed walk in $M$ must contain $a$, $b$, $c$, or $d$. It is straightforward (but tedious) to verify that there is no closed directed walk in $M$ containing $a$, $b$, $c$, or $d$; we show the argument for $b$. All edges in $M$ incident to  $b$ are implied by items (ii), (vii), and (xii). The only outgoing edge from $b$ in $M$ is implied by (vii) (see formulas $\phi_R^1$, $\phi_S^1$, and $\phi_T^1$) and every directed walk of positive length starting in $b$ is contained in the $S$-copy and is not closed (see Figure~\ref{fig:S}). It follows that $b$ is not contained in a closed directed walk in $M$. The argument for $a$, $c$, and $d$ is similar. Therefore, there is no closed directed walk in $M$, implying $\bA \models \neg \mu$.

\medskip

By Claim~1, there exists a homomorphism $h$ of $\bA$ into $\bC_\mu$, because $\bC_\mu$ is a dual of $\mu$. Let $a' := h(a)$, $b' := h(b)$, $c' := h(c)$, and $d' := h(d)$. 

\subparagraph*{Claim 2.} $(a',b') \in R$ and $(c',d') \notin R$.

On the one hand, since $h$ is a homomorphism and $(a,b) \in R^\bA$ (e.g., by item (ii)), $(a',b') \in R$. On the other hand, $(c',d') \notin R$: otherwise, since $\bA \models (\phi_R^0)^+(c,d,x^S, y^S)$ (item (vi)) and $h$ is a homomorphism, we get that
\[\bC_\mu \models \exists z,w (R(c',d') \wedge S(d',z) \wedge T(z,w) \wedge (\mu^*)^+(c',d',z,w)),\]
a contradiction with $\bC_\mu \models \neg \mu$.

\subparagraph*{Claim 3.} The cost of $\psi$ is 1 and $(a',b',c',d',c',d'), (c',d',a',b',c',d'), (c',d',c',d',a',b') \in \Opt(\psi)$.

By condition (ii), we have $\bA \models (\phi_R^1)^+(a,b,x^R, y^R)$. Note that since $h$ is a homomorphism, this implies that $\psi_R(a',b',h(x^R), h(y^R)) \leq 3$. Recall that the cost of $\psi_R$ is at least $3$, and,  therefore, $\psi_R(a',b',h(x^R), h(y^R)) = 3$ and $(a',b',h(x^R), h(y^R)) \in \Opt(\psi_R)$. Analogously, we have
\begin{align*}
\psi_S(c',d', h(y^R), h(z^R)) &= 5 \text{ and hence } (c',d', h(y^R), h(z^R))\in \Opt(\psi_S) \text{, and} \\
\psi_T(c',d',h(z^R), h(w^R)) &= 5 \text{ and hence } (c',d',h(z^R), h(w^R)) \in \Opt(\psi_T).
\end{align*}
By condition (i) and since $h$ is a homomorphism,
\begin{align} \label{eq:mu_R}
    R(h(x^R), h(y^R)) &+ S(h(y^R),h(z^R)) + T(h(z^R), h(w^R)) \nonumber \\
&+ \psi^*(h((x^R, y^R, z^R, y^R)^+)) \leq 1,
\end{align}
for the tuple $(x^R, y^R, z^R, y^R)^+$ completed as in item (i). Since 
 $\bC_\mu \models \neg \mu$, we get equality in \eqref{eq:mu_R}.
It follows that $\psi(a',b',c',d',c',d') \leq 1$ and since the cost of $\psi$ is $\geq 1$, we obtain that the cost of $\psi$ is equal to $1$ and $(a',b',c',d',c',d') \in \Opt(\psi)$. Analogously we argue that $(c',d',a',b',c',d')$ and $(c',d',c',d',a',b')$ lie in $\Opt(\psi)$.

\medskip
Since the cost of $\psi$ is $1$ (Claim~3), by the discussion under \eqref{eq:psi2}, $\Opt(\psi)$ contains only tuples $(d_1, \dots, d_6) \in D^6$ such that precisely one of $(d_1,d_2)$, $(d_3,d_4)$, and $(d_5,d_6)$ lies in $R$. Let $\bB = (D^2; \OIT^{\bB})$
be the relational structure where 
\[\OIT^{\bB}((x,y),(x',y'),(x'',y'')) := \Opt(\psi)(x,y,x',y',x'',y'').\]
Note that $\bB$ is a pp-power of $\Delta_\mu$.
We claim that $\bB$ is homomorphically equivalent to $(\{0,1\}; \OIT)$. Let $f \colon D^2 \to \{0,1\} $ be defined by $f(d_1, d_2) = 1$ if $(d_1, d_2)\in R$ and $f(d_1, d_2)=0$ otherwise. Then $f$ is a homomorphism from $\bB$ to $(\{0,1\}; \OIT)$ by the properties of $\Opt(\psi)$. For the other direction, let $g \colon \{0,1\} \to D^2$ be defined by $g(1) = (a',b')$ and $g(0) = (c',d')$. By Claim 3, $g$ is a homomorphism from $(\{0,1\}; \OIT)$ to $\bB$.
It follows that $\Delta_\mu$ pp-constructs $(\{0,1\}; \OIT)$ as we wanted to prove.
\end{proof}

The following corollary generalizes Theorem~\ref{thm:sjf-cycle}. 

\begin{corollary}\label{cor:cycle-hard-ucq}
Let $\mu$ be a union of minimal connected pairwise non-equivalent conjunctive queries over a binary signature containing a conjunctive query $\mu_0$. If $\multigr(\mu_0)$ contains a cycle of length $\geq 3$, then $\Delta_\mu$ pp-constructs $(\{0,1\}; \OIT)$ and the resilience problem for $\mu$ is NP-complete.
\end{corollary}

\begin{proof}
Let $\mu'$ be a self-join-free union of connected conjunctive queries over a binary signature $\tau'$ such that $\mu =f(\mu')$ for some $\mu'$-injective $f \colon \tau' \to \tau'$; it exists by Lemma~\ref{lem:self-join-var}.
Let $\mu_0'$ be a~conjunctive query from $\mu'$ such that 
$\mu_0=f(\mu_0')$ 
and let $\tau_0' \subseteq\tau'$ be the signature of $\mu_0'$. 

By Theorem~\ref{thm:self-join-var}, $\Delta_\mu$ pp-constructs $\Delta_{\mu'}$. By Lemma~\ref{lem:sjf-ucq}, the $\tau_0'$-reduct 
of $\Delta_{\mu'}$ is equal to $\bB_0^1$ for some dual $\bB$ of $\mu_0'$. Since $\bB$ and $\bC_{\mu_0'}$ are both duals of $\mu_0'$, the valued structure $\bB_0^1$ is fractionally homomorphically equivalent to $\Delta_{\mu_0'} = (\bC_{\mu_0'})_0^1$ (Remark~\ref{rem:fhom-eq}). 
Therefore,
$\Delta_{\mu'}$ pp-constructs $\Delta_{\mu_0'}$. By Theorem~\ref{thm:sjf-cycle}, $\Delta_{\mu_0'}$ pp-constructs $(\{0,1\}; \OIT)$. By the transitivity of pp-constructability, $\Delta_\mu$ pp-constructs $(\{0,1\}; \OIT)$. By Lemma~\ref{lem:hard}, 
$\VCSP(\Delta_\mu)$ is NP-hard. 
By Proposition~\ref{prop:connection}, the resilience problem for $\mu$ is NP-hard, and thus NP-complete. 
\end{proof}

\subsection{Hardness for queries with finite acyclic duals}
In this section we prove 
that the resilience problem for queries $\mu$ that have a non-trivial finite dual without directed cycles is NP-hard. We stress that this lemma crucially relies on our approach to analyse the complexity of the resilience problem for $\mu$ using the dual structure $\bC_\mu$ and $\Delta_\mu$.

\begin{lemma} \label{lem:long-path}
Let $\mu$ be a union of conjunctive queries over the signature $\{R\}$ such that the domain of $\bC_\mu$ is finite. Assume that $\bC_\mu$ contains at least one edge and does not contain any directed cycles. Then $\Delta_\mu$ pp-constructs $(\{0,1\}; \OIT)$.
\end{lemma}

\begin{proof}
Let $C$ be the domain of $\bC_\mu$. Let $k$ be the length of the longest directed path in $\bC_\mu$; it exists, because $\bC_\mu$ is finite and does not contain any directed cycles.
Let $\phi(x_0,x_1)$ be the pp-expression 
\[\inf_{x_2, \dots x_k} \big (R(x_0, x_1) + \Opt(R)(x_1, x_2) \dots + \Opt(R)(x_{k-1},x_k)\big ).\]
Let $(x,y)\in C^2$. Then $\phi^{\Delta_\mu}(x,y)=0$ if and only if there is a directed path in $\bC_\mu$ of length $k$ starting with the edge $(x,y)$.
If $\phi^{\Delta_\mu}(x,y)\neq 0$, then $\phi^{\Delta_\mu}(x,y) =1 $ if and only if there is a directed path in $\bC_\mu$ of  length $k-1$ starting in $y$. Finally, if $\phi^{\Delta_\mu}(x,y) \notin \{0,1\}$, then $\phi^{\Delta_\mu}(x,y) = \infty$.
Let $\Gamma$ be the valued $\{R\}$-structure on the domain $C$ where $R^\Gamma(x,y)= \phi^{\Delta_\mu}(x,y)$ for all $(x,y) \in C^2$. Note that $R^\Gamma \in \langle {\Delta_\mu} \rangle$.

Recall the valued structure $\Gamma_{\textup{MC}}$ from Example~\ref{expl:max-cut-vs}. Let $a,b \in C$ be such that there is a directed path in $\bC_\mu$ of length $k$ starting with the edge $(a,b)$. Let $f \colon \{0,1\} \to C$ be defined by $f(0):=a$ and $f(1):=b$. It is straightforward to verify that $\omega_f$ defined by $\omega_f(f)=1$ is a  
fractional homomorphism from $\Gamma_{\textup{MC}}$ to $\Gamma$.
Let $g \colon C \to \{0,1\}$ be defined by $g(x)=0$ for every $x \in C$ such that there is a directed path of length $k$ starting in $x$ and $g(x)=1$ otherwise. We argue that $\omega_g$ defined by $\omega_g(g)=1$ is a  
fractional homomorphism from $\Gamma$ to $\Gamma_{\textup{MC}}$. Let $(x,y) \in C^2$. Note that if $R^\Gamma(x,y) \geq 1$, then trivially $R^{\Gamma_{\textup{MC}}}(g(x), g(y)) \leq R^\Gamma(x,y)$. Suppose therefore that $R^\Gamma(x,y)=0$. Then by the definition of $\phi$, there is a directed path in $\bB$ of length $k$ starting with the edge $(x,y)$. By the definition of $g$, we have $g(x)=0$. Since there is no directed path of length $k+1$ in $\bC_\mu$ and $\bC_\mu$ does not contain directed cycles, there is no directed path of length $k$ starting in $y$ and therefore $g(y)=1$. Hence,  $R^{\Gamma_{\textup{MC}}}(g(x), g(y)) \leq R^\Gamma(x,y)$. It follows that $\omega_g$ is a 
fractional homomorphism from $\Gamma$ to $\Gamma_{\textup{MC}}$.

By the previous paragraph, $\Gamma$ is fractionally homomorphically equivalent to $\Gamma_{\textup{MC}}$. Since $R^\Gamma \in \langle \Delta_\mu \rangle$, we have that $\Delta_\mu$ pp-constructs $\Gamma_{\textup{MC}}$. 
We have already mentioned in Example~\ref{expl:max-cut-hard} that $\Gamma_{\textup{MC}}$ pp-constructs $(\{0,1\}, \OIT)$. By the transitivity of pp-constructability, $\Delta_\mu$ pp-constructs $(\{0,1\}, \OIT)$. 
\end{proof}

\section{Proof of Theorem~\ref{thm:main-ucq}}
We are now ready to prove the main result of the paper.

\begin{proof}[Proof of Theorem~\ref{thm:main-ucq}]
Since $\Aut(\bC_\mu) = \Aut(\Delta_\mu)$, items (1) and (2) are mutually exclusive by Corollary~\ref{cor:disjoint}. 
Hence it is enough to prove that item (1) or item (2) holds.
Without loss of generality, we may assume that all queries in $\mu$ are pairwise non-equivalent, minimal and connected (see Lemma~\ref{lem:con}). In particular, the queries in $\mu$ are pairwise homomorphically incomparable (see Section~\ref{sect:cq}). 
By this assumption, if $\mu$ contains $\mu_\ell$ or $\mu_e$, then $\mu$ is equal to $\mu_\ell$ or to $\mu_e$, respectively, in which case item (1) holds by Lemma~\ref{lem:tract}. 
We may therefore assume that $\mu$ contains neither $\mu_\ell$ nor $\mu_e$.

If $\mu$ contains a conjunctive query $\mu_0$ such that $\multigr(\mu_0)$ contains a cycle of length $\geq 3$, then item (2) holds by Corollary~\ref{cor:cycle-hard-ucq}. Suppose that this is not the case. Then for every query $\nu$ in $\mu$, $\multigr(\nu)$ is a tree, or $\nu$ contains the atoms $R(x,y)$ and $R(y,x)$ for some variables $x \neq y$, in which case $\nu = \mu_c$ by the minimality of $\nu$. 
Note that every $\nu$ such that $\multigr(\nu)$ is a tree 
has a homomorphism to $\mu_c$, 
so $\mu=\mu_c$ whenever $\mu$ contains $\mu_c$.
In this case, item (1) holds by Lemma~\ref{lem:tract}. Suppose therefore that $\mu \neq \mu_c$ and hence, $\multigr(\nu)$ is a tree for every query $\nu$ in $\mu$. Then $\mu$ has a finite dual 
by~\cite{NesetrilTardif} (see also \cite[Theorem 8.7]{Resilience-VCSPs}).
Since $\bC_\mu$ is the model-complete core of this dual, it also has a finite domain, so it is a finite directed graph with an edge relation $R$. Note that $\bC_\mu$ contains at least one edge, because $\mu \neq \mu_e$. 
It is easy to see that every orientation of a tree (in particular, every $\nu$ in $\mu$)  maps homomorphically to every directed cycle;
thus, $\bC_\mu$ does not contain directed cycles. 
By Lemma~\ref{lem:long-path}, 
$\Delta_\mu$
pp-constructs $(\{0,1\}; \OIT)$. 
By Lemma~\ref{lem:hard} and Proposition~\ref{prop:connection}, the resilience problem for $\mu$ is NP-complete. Therefore, item (2) holds.
\end{proof}



\bibliography{global}

@STRING{LICS = {Proceedings of the Annual Symposium on Logic in Computer Science (LICS)} }

@STRING{STOC = {Proceedings of the Annual Symposium on Theory of Computing (STOC)} }

@preamble{"\def\cprime{$'$} "}

@article{FominEtAl20,
	author = {Fomin, Fedor V. and Golovach, Petr A. and Thilikos, Dimitrios M.},
	journal = {Theory of Computing Systems},
	number = {2},
	pages = {251--271},
	title = {On the Parameterized Complexity of Graph Modification to First-Order Logic Properties},
	volume = {64},
	year = {2020}}

@article{MottetPinskerCores,
  author    = {Antoine Mottet and
               Michael Pinsker},
  title     = {Cores over {Ramsey} structures},
  journal   = {Journal of Symbolic Logic},
  volume    =   86,
  number=1,
  pages={352-361},
  year=2021
}

@inproceedings{Resilience-VCSPs,
author = {Bodirsky, Manuel and Semani\v{s}inov\'{a}, {\v{Z}}aneta and Lutz, Carsten},
title = {The Complexity of Resilience Problems via Valued Constraint Satisfaction Problems},
year = {2024},
isbn = {9798400706608},
publisher = {Association for Computing Machinery},
address = {New York, NY, USA},
url = {https://doi.org/10.1145/3661814.3662071},
doi = {10.1145/3661814.3662071},
booktitle = {Proceedings of the 39th Annual ACM/IEEE Symposium on Logic in Computer Science},
articleno = {14},
numpages = {14},
location = {Tallinn, Estonia},
series = {LICS '24}
}

@preamble{
   "\def\cprime{$'$} "
}

@article{BKOPP-equations,
author={Libor Barto and Michael Kompatscher and Miroslav Ol\v{s}\'{a}k and Trung Van Pham and Michael Pinsker},
title={Equations in oligomorphic clones and the constraint satisfaction problem for $\omega$-categorical structures},
journal={Journal of Mathematical Logic},
volume=19,
number=2, 
pages={\#1950010},
year=2019
}

@incollection{Pol,
  author    = {Libor Barto and
               Andrei A. Krokhin and
               Ross Willard},
  title     = {Polymorphisms, and How to Use Them},
  booktitle = {The Constraint Satisfaction Problem: Complexity and Approximability},
  publisher = {Schloss Dagstuhl - Leibniz-Zentrum fuer Informatik},
  pages     = {1-44},
  year      = {2017}
}

@article{LatestResilience, author = {Makhija, Neha and Gatterbauer, Wolfgang}, title = {A Unified Approach for Resilience and Causal Responsibility with Integer Linear Programming (ILP) and LP Relaxations}, year = {2023}, issue_date = {December 2023}, publisher = {Association for Computing Machinery}, address = {New York, NY, USA}, volume = {1}, number = {4}, url = {https://doi.org/10.1145/3626715}, doi = {10.1145/3626715}, journal = {Proc. ACM Manag. Data}, month = dec, articleno = {228}, numpages = {27} }

@article{RPQ, author = {Amarilli, Antoine and Gatterbauer, Wolfgang and Makhija, Neha and Monet, Mika\"{e}l}, title = {Resilience for Regular Path Queries: Towards a Complexity Classification}, year = {2025}, issue_date = {May 2025}, publisher = {Association for Computing Machinery}, address = {New York, NY, USA}, volume = {3}, number = {2}, url = {https://doi.org/10.1145/3725245}, doi = {10.1145/3725245}, journal = {Proc. ACM Manag. Data}, month = jun, articleno = {108}, numpages = {18} }

@article{Resilience,
  author    = {Cibele Freire and
               Wolfgang Gatterbauer and
               Neil Immerman and
               Alexandra Meliou},
  title     = {The Complexity of Resilience and Responsibility for Self-Join-Free
               Conjunctive Queries},
  journal   = {Proc. {VLDB} Endow.},
  volume    = {9},
  number    = {3},
  pages     = {180--191},
  year      = {2015},
  url       = {http://www.vldb.org/pvldb/vol9/p180-freire.pdf},
  doi       = {10.14778/2850583.2850592},
  bibsource = {dblp computer science bibliography, https://dblp.org}
}

@inproceedings{NewResilience,
  author    = {Cibele Freire and
               Wolfgang Gatterbauer and
               Neil Immerman and
               Alexandra Meliou},
  title     = {New Results for the Complexity of Resilience for Binary Conjunctive
               Queries with Self-Joins},
  booktitle = {Proceedings of the 39th {ACM} {SIGMOD-SIGACT-SIGAI} Symposium on Principles
               of Database Systems, {PODS} 2020, Portland, OR, USA, June 14-19, 2020},
  pages     = {271--284},
  year      = {2020},
  doi       = {10.1145/3375395.3387647},
  bibsource = {dblp computer science bibliography, https://dblp.org}
}

@InProceedings{ChandraMerlin,
author = "Ashok K. Chandra and Philip M. Merlin",
title = "Optimal Implementation of Conjunctive Queries in Relational Data Bases", 
booktitle = {Proceedings of the Symposium on Theory of Computing (STOC)},
year = {1977},
pages = {77-90}}

@Article{NesetrilTardif,
  author = {J. Ne\v{s}et\v{r}il and C. Tardif},
  title = 	 {Duality theorems for finite structures},
  journal = 	 {Journal of Combinatorial Theory, Series B},
  year = 	 {2000},
  volume = {80},
  pages = {80-97}
}

@BOOK{Hodges,
author = {Wilfrid Hodges},
title = {A shorter model theory},
publisher = {Cambridge University Press},
address = {Cambridge},
year = {1997}}

@BOOK{Oligo,
author = {Peter J. Cameron},
title = {Oligomorphic permutation groups},
publisher = {Cambridge University Press},
address = {Cambridge},
year = {1990}}

@Book{GareyJohnson, 
author = {Michael Garey and David Johnson},
title = {A guide to {NP}-completeness}, 
PUBLISHER = {CSLI Press},
address = {Stanford},
YEAR = {1978}}

@article{VCSP-Galois,
  author    = {David A. Cohen and
               Martin C. Cooper and
               P{\'{a}}id{\'{\i}} Creed and
               Peter G. Jeavons and
               Stanislav \v{Z}ivn\'y},
  title     = {An Algebraic Theory of Complexity for Discrete Optimization},
  journal   = {{SIAM} J. Comput.},
  volume    = {42},
  number    = {5},
  pages     = {1915-1939},
  year      = {2013}
}

@article{FullaZivny,
  author    = {Peter Fulla and
               Stanislav \v{Z}ivn\'y},
  title     = {A {G}alois Connection for Weighted (Relational) Clones of Infinite Size},
  journal   = {{TOCT}},
  volume    = {8},
  number    = {3},
  pages     = {9:1-9:21},
  year      = {2016}
}

@inproceedings{ThapperZivny13,
  author    = {Johan Thapper and
               Stanislav \v{Z}ivn\'y},
  title     = {The complexity of finite-valued {CSP}s},
  booktitle = {Proceedings of the Symposium on Theory of Computing Conference ({STOC}), Palo Alto, CA,
               USA, June 1-4, 2013},
  pages     = {695-704},
  year      = {2013}
}

@Misc{BodAutomorphismGroups,
	author = {Manuel Bodirsky},
	year = 2023,
	title = {Automorphism Groups},
	note = {Course Notes, TU Dresden, \url{https://wwwpub.zih.tu-dresden.de/~bodirsky/Automorphism-Groups.pdf}}
}

@Misc{KnaeuerMaster,
author={Simon Kn{\"a}uer},
title={Constraint Satisfaction over the Random Tournament},
  year      = {2018},
note={Master Thesis at the Institute of Algebra, TU Dresden}
}

@article{42,
author = {Manuel Bodirsky and Michael Pinsker and Andr\'{a}s Pongr\'acz},
title = {The 42 reducts of the random ordered graph},
journal={Proceedings of the LMS},
volume = {111},
doi = {10.1112/plms/pdv037},
number = {3},
pages = {591-632}, 
note = {Preprint available from arXiv:1309.2165},
year=2015
}

@Article{Cores-journal,
   author =       "Manuel Bodirsky",
   title = "Cores of countably categorical structures",
   journal = "Logical Methods in Computer Science ({LMCS})",
   volume    = {3},
   pages = {1-16},
  number    = {1},
   year = {2007}}

@Book{Book,
author = {Manuel Bodirsky},
title  = {Complexity of Infinite-Domain Constraint Satisfaction},
year   = {2021},
doi = {10.1017/9781107337534}, 
publisher = {Cambridge University Press},
series = {Lecture Notes in Logic (52)}}

@article{ViolaZivny,
  author       = {Caterina Viola and
                  Stanislav Zivn{\'{y}}},
  title        = {The Combined Basic {LP} and Affine {IP} Relaxation for Promise {VCSP}s
                  on Infinite Domains},
  journal      = {{ACM} Trans. Algorithms},
  volume       = {17},
  number       = {3},
  pages        = {21:1--21:23},
  year         = {2021},
  url          = {https://doi.org/10.1145/3458041},
  doi          = {10.1145/3458041},
  bibsource    = {dblp computer science bibliography, https://dblp.org}
}

@article{SchneiderViola,
title = {An application of {F}arkas' lemma to finite-valued constraint satisfaction problems over infinite domains},
journal = {Journal of Mathematical Analysis and Applications},
volume = {517},
number = {1},
pages = {126591},
year = {2023},
issn = {0022-247X},
doi = {10.1016/j.jmaa.2022.126591},
url = {https://www.sciencedirect.com/science/article/pii/S0022247X22006059},
author = {Friedrich Martin Schneider and Caterina Viola},
}

@phdthesis{ViolaThesis,
author={Caterina Viola},
title={Valued Constraint Satisfaction Problems over Infinite Domains},
school={TU Dresden},
year=2020
}

@phdthesis{ThesisZaneta,
author = {\v{Z}aneta Semani\v{s}inov\'{a}},
title  = {{V}alued {C}onstraint {S}atisfaction in
{S}tructures with an {O}ligomorphic
{A}utomorphism {G}roup},
year   = {2025},
address      = {Dresden, Germany},
note         = {Available at \url{
https://nbn-resolving.org/urn:nbn:de:bsz:14-qucosa2-974737}},
school       = {TU Dresden},
type         = {Dissertation}
}

@article{RDM,
author = {Meliou, Alexandra and Gatterbauer, Wolfgang and Suciu, Dan},
title = {Reverse data management},
year = {2011},
issue_date = {August 2011},
publisher = {VLDB Endowment},
volume = {4},
number = {12},
issn = {2150-8097},
url = {https://doi.org/10.14778/3402755.3402803},
doi = {10.14778/3402755.3402803},
journal = {Proc. VLDB Endow.},
month = aug,
pages = {1490 - 1493},
numpages = {4}
}

\end{document}